\documentclass[10pt]{article}

\usepackage[square,numbers]{natbib}

\usepackage[utf8]{inputenc}   
\usepackage[T1]{fontenc}  

\usepackage{pgf,tikz}
\usetikzlibrary{arrows}
\usepackage{multicol}
\usepackage{amsthm}
\usepackage{amsmath}
\usepackage{amssymb}
\usepackage{amsfonts}
\usepackage{stmaryrd}
\usepackage{mathabx}
\usepackage{latexsym}
\usepackage{color}
\usepackage{graphics,graphicx,graphpap}
\usepackage{multirow}
\usepackage{rotating}
\usepackage[new]{old-arrows}
\usepackage[all]{xy}
\usepackage[letterpaper]{geometry}
\usepackage{subfigure}
\usepackage{hyperref}
\usepackage{orcidlink}
\usepackage{authblk}
\theoremstyle{plain}

\usepackage{tocbibind}

\DeclareMathAlphabet{\mathpzc}{OT1}{pzc}{m}{it}
\newtheorem{theorem}{Theorem}
\newtheorem{prop}[theorem]{Proposition}
\newtheorem{cor}[theorem]{Corollary}
\newtheorem{lem}[theorem]{Lemma}

\newtheorem{que}[theorem]{Question}
\newcommand{\aste}[1]{\stackrel{*}{#1}}
\newcommand{\hocolim}{\operatornamewithlimits{\mathrm{hocolim}}}
\newcommand{\holim}{\operatornamewithlimits{\mathrm{holim}}}
\newcommand{\colim}{\operatornamewithlimits{\mathrm{colim}}}
\newcommand{\conn}{\operatornamewithlimits{\mathrm{conn}}}
\newcommand*{\email}[1]{%
    \normalsize\href{mailto:#1}{#1}\par
    }

\newtheorem{propapp}{Proposition}[section]
\newtheorem{corapp}{Corollary}[section]
\newtheorem{lemapp}{Lemma}[section]

\title{Polyhedral joins and graph complexes}
\author{Andr\'es Carnero Bravo\;\orcidlink{0009-0001-0221-5908}}
\affil{Centro de Ciencias Matemáticas, UNAM,\\ A.P. 61-3, Xangari, Morelia, Michoacán 58089, México\\ \email{carnero@matmor.unam.mx}}

\begin{document}
\maketitle
\begin{abstract}
    We give a decomposition of the suspension of a polyhedral join in terms of the polyhedral smash product of the 
    suspension of the family of pairs, and study some cases in which the formula can be desuspended, particularly for polyhedral joins over 
    independence complexes of graphs. We also give some lower bounds for the connectivity of polyhedral joins. We use these results to study 
    the homotopy type of the forest filtration for some lexicographic products of graphs.
\end{abstract}
\tableofcontents
\textbf{Acknowledgments.} This work was supported by UNAM Posdoctoral Program (POSDOC).
The author wishes to thank Omar Antolín-Camarena for his comments and suggestions which improved this paper. 
\section{Introduction}
Given two graphs $G$ and $H$ one can define their lexicographic product as the graph obtained by taking a copy of $H$ for each vertex of 
$G$ and adding all the possible edges between two copies if the corresponding vertices are adjacent in $G$. This construction seems 
natural and one can ask for analogous constructions for simplicial complexes; the polyhedral join is a natural generalization (see 
Section $2$ for definitions) ---these spaces have been studied before \citep{ayzenberg,okurapoljoin,vidaurre}, and are variants of
polyhedral products and polyhedral smash products, which have been 
extensively studied, see \citep{bahripolprodfeathomtheo,gitler,panovtorictopology} and the references within. 
Naturally one can ask if there is a connection of some kind between these objects and the answer is yes,
certain polyhedral joins play a fundamental role in the homotopy type of the suspension of 
some simplicial complexes associated to a lexicographic product of graphs. For this we will need to study the homotopy type of the 
suspension of a polyhedral join. 

The main result we prove about polyhedral joins is Theorem \ref{teosmashjoinpol} which says that
$$\Sigma\aste{Z}_K\left(\underline{X},\underline{A}\right)\simeq\hat{Z}_K\left(\underline{\Sigma X},\underline{\Sigma A}\right).$$
This theorem is (nearly) a generalization of the classical homotopy equivalence $X*Y \simeq \Sigma X \wedge Y$; more precisely,
that equivalence implies that $\Sigma(X*Y) \simeq (\Sigma X) \wedge (\Sigma Y)$, and our main result specializes to this when $K = S^0$.
This result also serves to link polyhedral joins to the more widely studied polyhedral products and polyhedral smash products.
For example, another classical homotopy equivalence, namely $\Sigma\left(X\times Y\right)\simeq\Sigma\left(X\wedge Y\right)\vee\Sigma X\vee\Sigma Y$, has also been generalized: in \cite{gitler}, it is show that if $K$ is a simplicial complex with 
vertex set $\underline{n}$, then
$$\Sigma Z_K\left(\underline{X},\underline{A}\right)\simeq\Sigma\left(\bigvee_{I\subset\underline{n}}\hat{Z}_{K_I}\left(\underline{X},\underline{A}\right)\right),$$
which reduces to the classical equivalence when $K=\Delta^1$. 

In \citep{gitler} it is shown that if $A_i\longhookrightarrow X_i$is null-homotopic for all $i$, then 
\begin{equation}\tag{$i$}\label{forsmsh}
\hat{Z}_K\left(\underline{X},\underline{A}\right)\simeq\bigvee_{\sigma\in K}\mathrm{lk}(\sigma)*\hat{D}(\sigma).
\end{equation}
This result, together with Theorem \ref{teosmashjoinpol}, give us a decomposition to the suspension of a polyhedral join.
It is natural to ask: When are these formulas are valid without the suspensions? For polyhedral products this question has been studied for particular 
families of pairs, namely $(\underline{CX},\underline{X})$ and particular families of simplicial complexes (see
\citep{grbicshift,welkerdecomposable,iriyedecompositionsshift,iriyefatwedge}). Here we study some families of graphs for 
which the formula can be desuspended when taking the polyhedral joins over their independence complex.

Our main applications of Theorem \ref{teosmashjoinpol} are to the study of some simplicial complexes associated to graphs. 
Given a graph $G$ and any $d \in \mathbb{N} \cup \{\infty\}$, we associated in \cite{forestfilt} a simplicial complex
$\mathcal{F}_d(G)$ whose vertices are the same as those of $G$ and where a subset $S$ is a simplex if the induced subgraph
on $S$ is a forest with maximal degree at most $d$ (for $d = \infty$ all degrees are allowed). This family of simplicial complexes
is a natural generalization of the independence complex of a graph, which in this notation is $\mathcal{F}_0(G)$. It is worth remarking
that because of the difficulty of calculating the homotopy type of the independence complex most of the literature has focused on 
specific families of graphs (see for example \cite{indcomplcartprod,Bousquet_M_lou_2007,homotopygoyal,Iriye_2012,MR4587515}). 
Also, the case $d=1$ is also called the $2$-independence complex \citep{salvetti2015,salvetti2018}-- where it was used, 
along with other complexes, to study the local homology of Artin groups. In \citep{forestfilt} the case 
$d=\infty$ was use to give an upper bound of the decycling number of a graph.

For the  lexicographic product $G\circ H$ we obtain the following formula (we consider the empty set to be a 
simplex):
$$\Sigma \mathcal{F}_0(G\circ H)\simeq\bigvee_{\sigma\in \mathcal{F}_0(G)}\sum\left(\mathcal{F}_0\left(G-\bigcup_{v\in\sigma}N[v]\right)*\mathcal{F}_0(H)^{*|\sigma|}\right).$$

In addition to this result about $\mathcal{F}_0$, for the family of complete multipartite graphs we study 
the whole forest filtration, computing its homotopy type for the star $K_{1,n}$ in Proposition \ref{propk1n} as a first step and 
for its suspension in Theorems \ref{teomultinf} and \ref{teomultford}.

The main tactic we will follow will be to give a decomposition of the complex studied into two or more 
subcomplexes, which expresses the complex as the colimit of a cofibrant punctured cube. This will allow us to take 
the homotopy colimit of the punctured cube.  

The structure of the paper is the following: in Section $2$ we give the preliminaries needed form graph theory and homotopy 
colimits; in Section $3$ we define polyhedral joins as well as polyhedral smash products, and we prove 
Theorem \ref{teosmashjoinpol} as well as some other results that are used in the next sections; in Section $4$ we study families of graphs for which
the formula given by Theorem \ref{teosmashjoinpol}  and (\ref{forsmsh}) can be desuspended; lastly, in Section $5$, we study the homotopy 
type of the graph complexes of the forest filtration of a lexicographic product.

\section{Preliminaries}
The graphs will be simple, no loops nor multiedges. For a graph $G$ its vertex set will be $V(G)$ and $E(G)$ will be its edge set. 
$\Delta(G)$ is maximal degree of the graph. For a vertex $v$, $N_G(v)=\{u\in V(G):\;uv\in E(G)\}$ is its open neighborhood 
and $N_G[v]=N_G(G)\cup\{v\}$ its closed neighborhood, 
we omit the subindex $G$ if there is no risk of confusion.
For $S\subseteq V(G)$, $G[S]$ is the graph induced by the set $S$. A forest is a graph 
without cycles. For all the graph definitions not stated here we follow \citep{graphsanddigraphs}.
The \textit{lexicographic product} of the graphs $G$ and $H$ is the graph $G\circ H$ 
with vertex set $V(G)\times V(H)$ and edge set 
$$\{\{(u,v_1),(u,v_2)\}:\;\{v_1,v_2\}\in E(H)\}\cup\{\{(u_1,v_1),(u_2,v_2)\}:\;\{u_1,u_2\}\in E(G)\}$$

A simplicial complex $K$ is a family of subsets of an finite set $V(K)$, 
the vertices of the complex, such that if $\tau \subseteq \sigma$ and $\sigma\in K$, then $\tau\in K$-- we wanna remark that we take the 
empty set as a simplex and allow ghost vertices as is standard while working with polyhedral products. Given a simplicial complex $K$ and 
a simplex $\sigma$, the link of $\sigma$ is the subcomplex 
$\mathrm{lk}(\sigma)=\{\tau\in K:\;\tau\cap\sigma=\emptyset\;\wedge\;\tau\cup\sigma\in K\}$ and its star is 
$\mathrm{st}(\sigma)=\{\tau\in K:\tau\cup\sigma\in K\}$. 
For a vertex we will write $\mathrm{lk}(v)$ and $\mathrm{st}(v)$ insted of $\mathrm{lk}(\{v\})$ or 
$\mathrm{st}(\{v\})$. The $q$-skeleton of a complex $K$, denoted $\mathrm{sk}_qK$, is the subcomplex of all the simplicies 
with at most $q+1$ elements. We will not distinguish between a complex and its geometric realization. 

Given a graph $G$ and a non-negative integer $d$, its \textit{$d$-forest complex} is the complex
$$\mathcal{F}_d(G)=\{\sigma\subseteq V(G):\:G[\sigma] \mbox{ is a forest such that }\Delta(G[\sigma])\leq d\};$$
for $d=\infty$ we take 
$$\mathcal{F}_\infty(G)=\{\sigma\subseteq V(G):\:G[\sigma] \mbox{ is a forest}\}.$$
For $d=0$, $\mathcal{F}_0(G)$ is the independence complex of $G$ and for $d=1$ is also called the $2$-independence 
complex \citep{salvetti2015}.

All product spaces will be taken with the compactly generated topology. Given a topological space $X$, $X^{\wedge n}$ will be the smash 
product of $n$ copies of $X$ and $X^{\ast n}$ will be the join of $n$ copies of $X$. For a family $(X_1,\ast_1),\dots,(X_n,\ast_n)$ 
of pointed CW-complexes, we take its \textit{fat wedge} as the space
$$W(X_1,\dots,X_n)=\left\lbrace(x_1,\dots,x_n)\in\prod_{i=1}^nX_i:\;x_i=\ast_i\mbox{ for at least one } i\right\rbrace$$

Taking $\underline{n}=\{1,\dots,n\}$, $\mathcal{P}(\underline{n})$ its power set and 
$\mathcal{P}_1(\underline{n})=\mathcal{P}(\underline{n})-\{\underline{n}\}$, a punctured $n$-cube $\mathcal{X}$ consists of:
\begin{itemize}
    \item a topological space $\mathcal{X}(S)$ for each $S$ in $\mathcal{P}_1(\underline{n})$, and
    \item a continuous function $f_{S\subseteq T}:\mathcal{X}(S)\longrightarrow\mathcal{X}(T)$ for each $S\subseteq T$,
\end{itemize}
such that $f_{S\subseteq S}=1_{\mathcal{X}(S)}$ and for any $R\subseteq S\subseteq T$ the following diagram commutes:
\begin{equation*}
    \xymatrix{
    \mathcal{X}(R) \ar@{->}[r]^{f_{R\subseteq S}} \ar@{->}[dr]_{f_{R\subseteq T}}& \mathcal{X}(S) \ar@{->}[d]^{f_{S\subseteq T}}\\
     & \mathcal{X}(T).
    }
\end{equation*}  
A punctured $n$-cube of interest for a given topological space $X$ is the constant punctured cube $\mathcal{C}_X$, 
where $\mathcal{C}_X(S)=X$ for any set $S$ and all the functions are $1_X$.
The colimit of a punctured $n$-cube is the space
$$\colim(\mathcal{X})=\bigsqcup_{S\in\mathcal{P}_1(\underline{n})}\mathcal{X}(S)/\sim,$$
where $\sim$ is the equivalence relation generated by $f_{S\subseteq T_1}(x_S)\sim f_{S\subseteq T_2}(x_S)$ for $T_1,T_2$ and $S\subseteq T_1,T_2$. From the definition
is clear that $\colim(\mathcal{C}_X)\cong X$ for any $X$.

For any $n\geq1$ and $S$ in $\mathcal{P}_1(\underline{n})$ we take:
$$\Delta(S)=\left\lbrace(t_1,t_2,\dots,t_n)\in \mathbb{R}^n:\;\sum_{i=1}^nt_i=1\mbox{ and }t_i=0\mbox{ for all }i\in S\right\rbrace,$$
and $d_{S\subseteq T}:\Delta(T)\longrightarrow\Delta(S)$ the corresponding inclusion. Now, for a punctured 
$n$-cube $\mathcal{X}$, the homotopy colimit is 
$$\hocolim(\mathcal{X})=\bigsqcup_{S\in\mathcal{P}_1(\underline{n})}\mathcal{X}(S)\times\Delta(S)/\sim,$$
where $(x_S,d_{S\subseteq T}(t))\sim(f_{S\subseteq T}(x_S),t)$. When $n=2$, we will specify a punctured $2$-cube via a diagram
\begin{equation*}
    \xymatrix{
    \mathcal{D}: & X\ar@{<-}[r]^{f} & Z \ar@{->}[r]^{g} & Y,
    }
\end{equation*}
and its homotopy colimit is called the homotopy pushout. 

Given a punctured $n$-cube $\mathcal{X}$ for $n\geq2$ and defining the punctured $(n-1)$-cubes
$\mathcal{X}_1(S)=\mathcal{X}(S)$ and $\mathcal{X}_2(S)=\mathcal{X}(S\cup\{n\})$,
we have that (see \citep[Lemma 5.7.6]{cubicalhomotopy})
$$\hocolim(\mathcal{X})\cong \hocolim\left(\mathcal{X}\left(\underline{n-1}\right)\longleftarrow \hocolim(\mathcal{X}_1)\longrightarrow \hocolim(\mathcal{X}_2)\right).$$

If for all  $S\subsetneq[n]$ the map 
$$\colim_{T\subsetneq S} X_T\longrightarrow X_S$$
is a cofibration, we call the punctured cube cofibrant.  
If we have CW-complexes $X_1,\dots,X_n$ such that their intersections are subcomplexes and we take the punctured cube given by the 
intersections and the inclusions, then the punctured cube is cofibrant and 
$\hocolim(\mathcal{X})\simeq \colim(\mathcal{X})$ (see \citep[Proposition 5.8.25]{cubicalhomotopy}).

In particular combining the last two observations, we see that we can compute the homotopy type of a union of the CW complexes $X,Y,Z$ that intersect in subcomplexes, by means of three homotopy pushouts, as shown in the following diagram whose top and bottom squares, as well as the rightmost vertical square, are homotopy pushouts and where $R \simeq X \cup Y \cup Z$:

\begin{equation*}
\xymatrix{
X \cap Y \cap Z \ar@{->}[rr] \ar@{->}[dr] \ar@{->}[dd] & & Y \cap Z \ar@{-}[d] \ar@{->}[rd] & & \\
 & X \cap Z \ar@{->}[rr] \ar@{->}[dd] & \ar@{->}[d] & P \ar@{->}[r] \ar@{->}[dd] & Z \ar@{->}[dd] \\
X \cap Y \ar@{-}[r] \ar@{->}[dr] & \ar@{->}[r] & Y \ar@{->}[dr] &  & \\
  & X \ar@{->}[rr] &  & Q \ar@{->}[r] & R
}
\end{equation*}

Given two punctured $n$-cubes $\mathcal{X},\mathcal{Y}$, a map between them is a collection of maps 
$$f_S:\mathcal{X}(S)\longrightarrow\mathcal{Y}(S)$$ 
for each $S$ such that for any pair of sets $S\subseteq T$ the following square commutes:
\begin{equation*}
    \xymatrix{
    \mathcal{X}(S) \ar@{->}[r] \ar@{->}[d] & \mathcal{X}(T) \ar@{->}[d]\\
    \mathcal{Y}(S) \ar@{->}[r] & \mathcal{Y}(T)
    }
\end{equation*}
If each $f_S$ is a homotopy equivalence, we say that the map is a \textit{homotopy equivalence of cubes}, this is 
justified by the fact that the corresponding homotopy colimits are homotopy equivalent (see \citep[Theorem 5.7.8]{cubicalhomotopy}). While 
in general we do require that the corresponding squares commute, for push-out diagrams we only need that the squares to commute up to homotopy 
(see \citep[Theorem 6.2.8]{introhomo}) ---from this is easy to prove the following lemma.

\begin{lem}\label{homocolimpegado}
Let $X,Y,Z$ be spaces with maps $f:Z\longrightarrow X$ and $g:Z\longrightarrow Y$ such that both maps are null-homotopic. Then
$$\hocolim\left(\mathcal{S}\right)\simeq X\vee Y\vee \Sigma Z$$
where 
\begin{equation*}
    \xymatrix{
    \mathcal{S}: & Y \ar@{<-}[r]^{g} & Z \ar@{->}[r]^{f} & X
    }
\end{equation*}
\end{lem}

While most of the homotopy colimits we will use are over cubical diagrams, we will need more general 
homotopy colimits for a couple of results. Taking a poset $(\mathcal{P},\leq)$ as category and a functor 
$\mathcal{X}:\mathcal{P}\longrightarrow Top$, its homotopy colimit is the space
$$\hocolim_{\mathcal{P}}\mathcal{X}=\mathrm{coeq}\left(\xymatrix{
\displaystyle\coprod_{i\leq j}\mathcal{X}(i)\times|(j\downarrow\mathcal{P})^{op}| \ar@<1ex>[r]\ar@{->}[r]& \displaystyle\coprod_{i}\mathcal{X}(i)\times|(i\downarrow\mathcal{P})^{op}|
}\right)$$
where the maps are induced by the maps $\mathcal{X}(i\leq j)\times 1$ and $1\times|(j\downarrow\mathcal{P})^{op}\rightarrow(i\downarrow\mathcal{P})^{op}|$, and 
$|(i\downarrow\mathcal{P})^{op}|$ is the nerve of the opposite of the under category of $i$. This construction is known as the Bousfield-Kan formula. 
Notice that this definition agrees with 
the definition for homotopy colimits of punctured cubes. Homotopy comlimts for other diagrams are defined in a similar fashion, see 
\citep{dugger2008primer,cubicalhomotopy,strom}. If $\mathcal{I}$ is a small category, then its \textit{nerve} is homeomorphic to the homotopy 
colimit of the constant functor $C_*$, \textit{i.e.} $|\mathcal{I}|\cong\hocolim_\mathcal{I}C_*$ \citep[see Example 4.1]{dugger2008primer}.

Giving a functor $F:\mathcal{I}\longrightarrow\mathcal{J}$ 
and an object $j$ in $\mathcal{J}$, the category $(j\downarrow F)$ is the category with objects $[i,f:j\rightarrow F(i)]$ with $i$ in 
$\mathcal{I}$ and, given $[i,f:j\rightarrow F(i)]$ and $[i',g:j\rightarrow F(i')]$, a map between them is a map $\varphi:i\rightarrow i'$ 
such that $F(\varphi)\circ f=g$. We say $F$ is \textit{homotopy terminal} if $|(j\downarrow F)|$ is non-empty and contractible for all $j \in \mathcal{J}$.

\begin{prop}\label{propnathomsmjoin}
Let $X_1,\dots,X_n$ be CW-complexes, with $n\geq2$. Then there is a space $Z$ and natural 
homotopy equivalences
$$\bigwedge_{i=1}^n\Sigma X_i\longleftarrow Z\longrightarrow\Sigma\left(\bigast_{i=1}^nX_i\right)$$
where the base points are $[\{0\}\times X_i]$.
\end{prop}
\begin{proof}
We take $\mathcal{I}=\mathcal{P}_1(\underline{2})$ and $\mathcal{I}^n=\mathcal{I}\times\cdots\times\mathcal{I}$ the category given by 
the product of $n$ copies of $\mathcal{I}$. For $n\geq2$ we construct the category $\mathcal{J}_n$ by taking 
$\mathcal{P}_1(\underline{n})$, adding two new objects $0$ and $1$ and adding all possible arrows from $\mathcal{P}_1(\underline{n})$ to 
the new objects. 

We define 
$$\mathcal{Y}(x)=\left\lbrace\begin{array}{cc}
   \{0\}  & \mbox{ if } x=0\\
   \{1\}  & \mbox{ if } x=1\\
   \displaystyle\prod_{i\notin x}X_i 
\end{array}\right.$$
where the maps in $\mathcal{P}_1(\underline{n})$ are sent to the corresponding projections and the rest are sent to the corresponding 
constant maps. Then 
$$\hocolim_{\mathcal{J}_n}\mathcal{Y}\cong\Sigma\left(\bigast_{i=1}^nX_i\right).$$

Now, we define $\mathcal{X}:\mathcal{I}^n\longrightarrow Top$ as $\mathcal{X}(S_1,\dots,S_n)=B_1\times\cdots\times B_n$, 
where for each $i:$
$$B_i=\left\lbrace\begin{array}{cc}
   \{0\}  & \mbox{ if } S_i=\{2\}\\
   \{1\}  & \mbox{ if } S_i=\{1\}\\
    X_i & \mbox{ if } S_i=\emptyset
\end{array}\right.$$
and each maps is sent to the product of identities and constant maps accordingly to if the domain and codomain are the same or not. Then
$$\hocolim_{\mathcal{I}^n}\mathcal{X}\cong\prod_{i=1}^n\Sigma X_i$$
Now, we take $\mathcal{W}_n$ the full subcategory of $\mathcal{I}^n$ with objects $(S_1,\dots,S_n)$ such that $S_i=\{2\}$ for some
$i$, and take $F:\mathcal{W}_n\longhookrightarrow\mathcal{I}^n$ the inclusion. For all $n\geq2$ 
$$\hocolim_{\mathcal{W}_n}\mathcal{X}\circ F\cong W(\Sigma Y_1,\dots,\Sigma Y_n).$$
therefore
$$\left(\Sigma X_1\right)\wedge\cdots\wedge\left(\Sigma X_n\right)=\mathrm{cofiber}\left(\hocolim_{\mathcal{W}_n}\mathcal{X}\circ F\longhookrightarrow\hocolim_{\mathcal{I}^n}\mathcal{X}\right)$$

We affirm that $|\mathcal{W}_n|\simeq*$. This follows from
$$|\mathcal{W}_n|\cong\hocolim_{\mathcal{W}_n}C_*\cong W(\Sigma*,\dots,\Sigma*)\simeq*$$

Now, we define $G:\mathcal{I}^n\longrightarrow\mathcal{J}_n$ as 
$$G((S_1,\dots,S_n))=\left\lbrace\begin{array}{cc}
   0  & \mbox{ if } S_i=\{2\} \mbox{ for some } i\\
   1  & \mbox{ if } S_i=\{1\} \mbox{ for all } i\\
   S=\{i:\;S_i\neq\emptyset\} & \mbox{otherwise}
\end{array}\right.$$ 
 
Now, we show that the objects $|(j\downarrow G)|$ are contractible for any $j$. We take $j$ an object in $\mathcal{J}_n$.
If $j=0$, then $(0\downarrow G)\cong\mathcal{W}_n$, thus $|(0\downarrow G)|\simeq*$. If $j=1$, then $(1\downarrow G)$ has only one object and one map. Thus $|(1\downarrow G)|\simeq*$.

Assume $j\neq0,1$, then there is only one object $i_j$ in $\mathrm{Obj}(\mathcal{I}^n)$ such that $G(i_j)=j$. Now, we take $W(j)$ the full 
subcategory of $(j\downarrow G)$ which is given by the objects of the form $[w,j\rightarrow0]$ where $w$ is in $\mathrm{Obj}(\mathcal{W}_n)$. 
Then $W(j)\cong\mathcal{W}_n$ and $|W(j)|\simeq*$. Next, we take $A(j)$ the full subcategory of $(j\downarrow G)$ given by objects 
$[i,j\rightarrow G(i)]$ where $[i,i_j\rightarrow i]$ is an object of $(i_j\downarrow\mathcal{I}^n)$. 
By construction there is a bijection between  $\mathrm{Obj}((i_j\downarrow\mathcal{I}^n))$ and 
$\mathrm{Obj}(A(j))$, there is also a bijection between the maps. Thus $A(j)\cong(i_j\downarrow\mathcal{I}^n)$ and $|A(j)|\simeq*$.
Lastly, we take $B(j)$ the full subcategory given by the objects that are objects of $W(j)$ and $A(j)$ simultaneously. We get that 
$$|(j\downarrow G)|\cong\hocolim_{(j\downarrow G)}C_*\cong\hocolim\left(\hocolim_{A(j)}C_*\longhookleftarrow\hocolim_{B(j)}C_*\longhookrightarrow\hocolim_{W(j)}C_*\right)$$
$$\cong\hocolim\left(|W(j)|\longhookleftarrow|B(j)|\longhookrightarrow|A(j)|\right)\simeq\Sigma|B(j)|$$
Notice that the objects of $B(j)$ are given by an object $w$ of $\mathcal{W}_n$ such that there is a map $i_j\rightarrow w$ in 
$\mathrm{Hom}_{_{\mathcal{I}^n}}(i_j,w)$-- by constructions there is only one map or non.  
Now, we define $\rho(i_j)$ as the number of entries in $i_j$ that are equal to $\emptyset$. If $\rho(i_j)=1$, then $B(j)$ only 
have one object. We assume $\rho(i_j)\geq2$. Given a object $[w,j\rightarrow0]$ in $B(j)$, it is determined by $i_j$. To see this, we take  
$l_1<\cdots<l_{\rho(i_j)}$ the entries of $i_j$ which are equal to $\emptyset$, then $w$ is determined by these entries as the rest are 
equal to $\{1\}$. Thus $B(j)\cong\mathcal{W}_{\rho(i_j)}$. In any case, $|B(j)|\simeq*$.

Therefore $G:\mathcal{I}^n\longrightarrow\mathcal{J}_n$ is homotopy terminal, this imply that the 
induced natural map 
$$\hocolim_{\mathcal{I}^n}(\mathcal{Y}\circ G)\longrightarrow\hocolim_{\mathcal{J}_n}\mathcal{Y}$$
is a (weak) homotopy equivalence (see  \citep[Theorem 6.7]{dugger2008primer}). 
The space $Z$ of the statement is $\hocolim_{\mathcal{I}^n}(\mathcal{Y}\circ G)$. Thus only rest to show one homotopy equivalence. 

Now, we define 
$\alpha:\mathcal{X}\longrightarrow\mathcal{Y}\circ G$ with
$$\alpha_{(S_1,\dots,S_n)}:\mathcal{X}((S_1,\dots,S_n))\longrightarrow\mathcal{Y}\circ G((S_1,\dots,S_n))$$ 
the constant function if $S_i=\{2\}$ for some $i$ and the obvious homeomorphism in other case. We obtain a map 
$$f_\alpha:\hocolim_{\mathcal{I}^n}\mathcal{X}\longrightarrow\hocolim_{\mathcal{I}^n}\mathcal{Y}\circ G$$
Next we take the following commutative diagram  
\begin{equation*}
    \xymatrix{
     \displaystyle\hocolim_{\mathcal{W}_n}\mathcal{X}\circ F\ar@{->}[rr] \ar@{^(->}[d]& & \displaystyle\hocolim_{\mathcal{W}_n}C_\ast \ar@{^(->}[d]\ar@{->}[rr]&& \ast\ar@{->}[d]\\
    \displaystyle\hocolim_{\mathcal{I}^n}\mathcal{X} \ar@{->}[rr]_{f_\alpha}&& \displaystyle\hocolim_{\mathcal{I}^n}\mathcal{Y}\circ G \ar@{->}[rr]_{g}&& \mathrm{P}
    }
\end{equation*}
where $\mathrm{P}$ is the push-out of the right square. By construction is clear that the left square is also a push-out.
Therefore, the external square is a push-out and this push-out is 
$$\mathrm{cofiber}\left(\hocolim_{\mathcal{W}_n}\mathcal{X}\circ F\longhookrightarrow\hocolim_{\mathcal{I}^n}\mathcal{X}\right)$$
Now, $\hocolim_{\mathcal{W}_n}C_\ast\simeq\ast$, thus the map $g$ is a homotopy equivalence (see \citep[Proposition 3.6.2]{cubicalhomotopy}).
\end{proof}

In the next section we will give some bounds for the connectivity of a polyhedral join. For this we will need a 
folklore result. We say that a connected CW $Y$ is $n$-\textit{truncated} if $\pi_q(Y)\cong0$ for all $q>n$ and any base point. 
The \textit{$n$-truncation or $n$-th Postnikov section} is a functor $\mathrm{P}_n:CW\longrightarrow CW$ 
such that for any connected CW-complex $X$ 
the space $\mathrm{P}_n X$ is $n$-truncated and there is a $(n+1)$-connected map $\Phi_n:X\longrightarrow\mathrm{P}_n X$-- we want the 
construction to be functorial and we do not need fibrations between consecutive sections as in the usual definition of the Postnikov tower. 
If $X=\emptyset$, then we take $\mathrm{P}_nX=\emptyset$. We will use the following proposition in next section, suprisingly we could not 
find a reference for this folklore result, for completeness we will provide a proof in the Appendix as well as the definition and basic 
properties of $\mathrm{P}_nX$.

\begin{prop}\label{postnikovhocolim}
Take $\mathcal{I}$ a small category and $\mathcal{X}:\mathcal{I}\longrightarrow Top$ such that 
$\mathcal{X}(i)$ is either empty or a connected CW-complex for any $i$ in $\mathcal{I}$. Then
$$\mathrm{P}_n\left(\hocolim_{i\in\mathcal{I}}\mathcal{X}(i)\right)\simeq\mathrm{P}_n\left(\hocolim_{i\in\mathcal{I}}\mathrm{P}_n(\mathcal{X}(i))\right)$$
\end{prop}

\section{Polyhedral products}
Given a family of pointed pairs of CW-complexes $(\underline{X},\underline{A})=\{(X_i,A_i)\}_{i=1}^n$ and $K$ a simplicial complex on 
$n$ vertices, we define the \textit{polyhedral smash product} determined by $(\underline{X},\underline{A})$ and $K$ as the space
$$\hat{Z}_K(\underline{X},\underline{A})=\bigcup_{\sigma\in K}\hat{D}(\sigma)$$
with
$$\hat{D}(\sigma)=\bigwedge_{i=1}^nY_i,\;\;\;\;\mbox{where }Y_i=\left\lbrace\begin{array}{cc}
    X_i & \mbox{ if }i\in\sigma \\
    A_i & \mbox{ if }i\notin\sigma.
\end{array}\right.$$

\begin{theorem}\label{polysmash}\citep{gitler}
Let $K$ be a complex with $n$ vertices and $(\underline{X},\underline{A})$ a family of pointed pairs of CW-complexes such that 
$A_i\longhookrightarrow X_i$ is null-homotopic. Then 
$$\hat{Z}_K(\underline{X},\underline{A})\simeq\bigvee_{\sigma\in K}\mathrm{lk}(\sigma)*\hat{D}(\sigma)$$
\end{theorem}

Given a complex $K$ with vertices $\underline{n}$ and $(\underline{X},\underline{A})=\{(X_1,A_1),\dots,(X_n,A_n)\}$ a family of 
pairs of CW-complexes, we define the \textit{polyhedral join} as the space
$$\aste{Z}_{K}(\underline{X},\underline{A})=\bigcup_{\sigma\in K}\aste{D}(\sigma)$$
with 
$$\aste{D}(\sigma)=\bigast_{i=1}^nY_i,\;\;\;\;\mbox{where }Y_i=\left\lbrace\begin{array}{cc}
    X_i & \mbox{ if }i\in\sigma \\
    A_i & \mbox{ if }i\notin\sigma
\end{array}\right.$$
With the polyhedral smash product all the spaces must be non-empty, but for the polyhedral join we will allow pairs
$(X,\emptyset)$, where $X$ will always be non-empty.

We state without proof the following lemma which follows from the definition of the polyhedral join.
\begin{lem}\label{poljoinovjoin}
Let $K_1,\dots,K_r$ complexes with disjoint vertex sets. Taking $\displaystyle K=\bigast_{i=1}^rK_i$, a family of CW pairs 
$\left(\underline{X^i},\underline{A^i}\right)$ for each $K_i$ and $\left(\underline{X},\underline{A}\right)$ the family formed from all 
these families, we have that
$$\aste{Z}_K\left(\underline{X},\underline{A}\right)\cong\bigast_{i=1}^r\aste{Z}_{K_i}\left(\underline{X^i},\underline{A^i}\right)$$
\end{lem}

Next proposition will give us a lower bound for the connectivity of the polyhedral join.
\begin{prop}\label{connpoljoin}
Let $K$ be a connected simplicial complex with $n\geq2$ vertices and $\left(\underline{X},\underline{A}\right)$ a family of CW-pairs.
\begin{enumerate}
    \item If $conn(X_i)\geq l\geq \mathrm{conn}(K)$ and $A_i=\emptyset$ for all $i$, then for all $r\leq l+1$ 
    $$\pi_r\left(\aste{Z}_K(\underline{X},\underline{\emptyset})\right)\cong\pi_r(K)$$
    \item If for all $i$, $X_i\neq\emptyset\neq A_i$, then $\conn\left(\aste{Z}_K(\underline{X},\underline{A})\right)\geq n-2$
     \item  If for all $i$, $\conn(X_i)>\conn(A_i)\geq-1$, then $\displaystyle\conn\left(\aste{Z}_K(\underline{X},\underline{A})\right)\geq2n-1+\sum_{i=1}^nconn(A_i)$
\end{enumerate}
\end{prop}
\begin{proof}
In any case, let $Y$ be the polyhedral join and take $r$ equal to the bound in each case. By Proposition \ref{postnikovhocolim}
$$\mathrm{P}_r(Y)\simeq \mathrm{P}_r\left(\hocolim_{\mathcal{K}}\aste{D}\right)\simeq\mathrm{P}_r\left(\hocolim_{\mathcal{K}}\mathrm{T}_r^\infty\left(\aste{D}\right)\right)$$
In the first case the only spaces with non-contractible truncations are 
the space associated to the empty set and the ones associated to the vertices, in the second case all the spaces have trivial 
truncations and in the third case only the space associated to the empty set has a non-contractible truncation. 
We can change the diagram in each case for one with a single point space for each contractible space and all the maps for the constant map.
Therefore, by the Wedge Lemma (see \citep[Proposition 3.5]{zieglerhocolim}), we have that:
\begin{itemize}
    \item For the first case, $\displaystyle\mathrm{P}_r(Y)\simeq\mathrm{P}_r\left(K\vee\bigvee_{i=1}^nK(\pi_r(X_i),r)*\mathrm{lk}_K(i)\right)$.
    \item For the second case, $\displaystyle\mathrm{P}_r(Y)\simeq\mathrm{P}_r(*)$.
    \item For the third case, $\displaystyle\mathrm{P}_r(Y)\simeq\mathrm{P}_r\left(K*\mathrm{P}_r\left(\bigast_{i=1}^nA_i\right)\right)$.
\end{itemize}
We arrive at the desired result for the second and third cases. For the first case, by Proposition \ref{postnikovhocolim}
$$\mathrm{P}_r\left(K\vee\bigvee_{i=1}^nK(\pi_r(X_i),r)*\mathrm{lk}_K(i)\right)\simeq$$
$$\mathrm{P}_r\left(\hocolim\left(\mathrm{P}_rK\longleftarrow\mathrm{P}_r(\ast)\longrightarrow\mathrm{P}_r\left(\bigvee_{i=1}^nK(\pi_r(X_i),r)*\mathrm{lk}_K(i)\right)\right)\right)\simeq$$
$$\mathrm{P}_r\left(\hocolim\left(\mathrm{P}_rK\longleftarrow\ast\longrightarrow \ast\right)\right)\simeq\mathrm{P}_r(\mathrm{P}_r K)\simeq\mathrm{P}_r(K)$$
\end{proof}

There is also a definition of polyhedral product similar to the polyhedral smash product or polyhedral join given above, but using the 
Cartesian product of spaces-- this is the space $Z_K(\underline{X},\underline{A})$ we mentioned in the introduction. 
In \citep{gitler} the homotopy type of the suspension of such a polyhedral product is given in terms of the polyhedral smash product; here 
we show that the suspension of a polyhedral join has the same homotopy type as a polyhedral smash product. 

\begin{theorem}\label{teosmashjoinpol}
If $(\underline{X},\underline{A})=\{(X_1,A_1),\dots,(X_n,A_n)\}$ is a family of  pairs of CW-complexes and 
$(\underline{\Sigma X},\underline{\Sigma A})=\{(\Sigma X_1,\Sigma A_1),\dots,(\Sigma X_n,\Sigma A_n)\}$ the family of their suspensions as pointed pairs, then
$$\Sigma\aste{Z}_K(\underline{X},\underline{A})\simeq\hat{Z}_K\left(\underline{\Sigma X},\underline{\Sigma A}\right).$$
\end{theorem}
\begin{proof}
If $\sigma_1,\dots,\sigma_r$ are the maximal simplexes of $K$, we take the punctured $r$-cube and define the diagram 
$\mathcal{Y}:\mathcal{P}_1(\underline{r})\longrightarrow Top$ as
$$\mathcal{Y}(S)=\bigcap_{i\in S^c}\aste{D}(\sigma_i)$$
with the inclusions as maps. Then 
$$\aste{Z}_{K}(\underline{X},\underline{A})\simeq \hocolim_{\mathcal{P}_1(\underline{r})}(\mathcal{Y})$$ 
and we have a weak-homotopy equivalence 
$\Sigma\hocolim_{\mathcal{P}_1(\underline{r})}(\mathcal{Y})\simeq\hocolim_{\mathcal{P}_1(\underline{r})}(\Sigma\mathcal{Y})$ 
(see \citep{cubicalhomotopy} Corollary 5.8.10). Now, we define the diagram 
$\hat{\mathcal{Y}}:\mathcal{P}_1(\underline{r})\longrightarrow Top$ as
$$\hat{\mathcal{Y}}(S)=\bigcap_{i\in S^c}\hat{D}(\sigma_i)$$
and again with the inclusions as maps. Thus 
$$\hat{Z}_K\left(\underline{\Sigma X},\underline{\Sigma A}\right)\simeq\hocolim_{\mathcal{P}_1(\underline{r})}\hat{\mathcal{Y}}$$

Now, by Proposition \ref{propnathomsmjoin} for each $S$ in $\mathcal{P}_1(\underline{n})$ there is a space $Z_S$ and natural homotopy 
equivalences
$$\hat{\mathcal{Y}}(S)\longleftarrow Z_S\longrightarrow\Sigma\mathcal{Y}(S)$$
If we define $\mathcal{Z}:\mathcal{P}_1(\underline{n})\longrightarrow Top$ as $\mathcal{Z}(S)=Z_S$, then the induced maps
$$\hocolim_{\mathcal{P}_1(\underline{r})}\hat{\mathcal{Y}}\longleftarrow\hocolim_{\mathcal{P}_1(\underline{r})}\mathcal{Z} \longrightarrow\hocolim_{\mathcal{P}_1(\underline{r})}\Sigma\mathcal{Y}$$
are homotopy equivalences.
\end{proof}

From the last theorem and Theorem \ref{polysmash} we get the following corollary:

\begin{cor}\label{corpolyjoinnullpairs}
If $(\underline{X},\underline{A})=\{(X_1,A_1),\dots,(X_n,A_n)\}$ is a family of  pairs of CW-complexes such that the inclusion 
$\Sigma A_i\longhookrightarrow \Sigma X_i$ is null-homotopic for all $i$, then 
$$\Sigma\aste{Z}_K(\underline{X},\underline{A})\simeq\bigvee_{\sigma\in K}\Sigma \mathrm{lk}(\sigma)\;\ast\aste{D}(\sigma).$$
\end{cor}

Now we will see an example where the equivalence of the last corollary can not be desuspended. For this, 
we take a non null-homotopic map $\alpha:\mathbb{S}^9\longrightarrow\mathbb{S}^5$ and its 
mapping cylinder $M_\alpha$. We known this maps exists because $\pi_9(\mathbb{S}^5)\cong\mathbb{Z}_2$ 
and we also know that its suspension is null-homotopic because $\pi_{10}(\mathbb{S}^6)\cong0$ (see \citep[Table I, p. 186]{Toda_1963}). 
Taking $K$ to be the disjoint union of 
two points and $(\underline{X},\underline{A})=\{(M_\alpha,\mathbb{S}^9),(\mathbb{D}^{10},\emptyset)\}$ we have that the polyhedral 
join over $K$ has the homotopy type of 
$Y=\hocolim\left(\mathbb{D}^{10}\longhookleftarrow\mathbb{S}^9\longrightarrow\mathbb{S}^5\right)$. From last corollary, we known that 
$\Sigma Y\simeq\mathbb{S}^{11}\vee\mathbb{S}^6$. 
By construction $\pi_9(Y)\cong0$,
therefore we cannot desupend the formula, \textit{i.e.} $Y\nsimeq\mathbb{S}^{10}\vee\mathbb{S}^5$. A natural question is: When can 
we desuspend the formula of Corollary \ref{corpolyjoinnullpairs}? We will see some cases where under the hypothesis of 
Theorem \ref{polysmash} we can desuspend the formula.

\begin{cor}\label{corpoljoincone}
Let $L$ be a simplicial complex with vertex set $V(L)=\{v_1,\dots,v_n\}$ and take 
$K$ its cone with $v_0$ as the apex vertex. If $(\underline{X},\underline{A})=\{(X_0,A_0),(X_1,A_1),\dots,(X_n,A_n)\}$ 
is a family of  pairs of CW-complexes such that for each $i>0$ the inclusion $\Sigma A_i\longhookrightarrow\Sigma X_i$ is null-homotopic or 
$A_i=\emptyset$  and $X_0$ is connected, then 
$$\aste{Z}_K(\underline{X},\underline{A})\simeq\bigvee_{\sigma\in L}\mathrm{lk}_L(\sigma)*\aste{D}(\sigma\cup\{v_0\}).$$
\end{cor}
\begin{proof}
We choose a vertex of $X_0$, then we have that 
$$\aste{Z}_K(\underline{X},\underline{A})=X_0*\aste{Z}_L(\underline{X},\underline{A})\simeq$$
$$X_0\wedge\Sigma\aste{Z}_L(\underline{X},\underline{A})\simeq X_0\wedge\left(\bigvee_{\sigma\in L}\Sigma \mathrm{lk}_L(\sigma)\;\ast\aste{D}(\sigma)\right)\simeq$$
$$\left(\bigvee_{\sigma\in L}X_0\wedge\Sigma\mathrm{lk}_L(\sigma)\;\ast\aste{D}(\sigma)\right)\simeq\bigvee_{\sigma\in L}\mathrm{lk}_L(\sigma)*\aste{D}(\sigma\cup\{v_0\}).$$
Now, for any simplex $\sigma$ that does not contain $v_0$, we have that $\mathrm{lk}_K(\sigma)\simeq*$ and  
$\mathrm{lk}_K(\sigma\cup\{v_0\})=\mathrm{lk}_L(\sigma)$.
\end{proof}

\begin{cor}\label{corsuspcompl}
Let $L$ be a simplicial complex with vertex set $V(L)=\{v_1,\dots,v_n\}$ and we take 
$K=\displaystyle\left(\bigvee_{l-1}\mathbb{S}^0\right)*L$ with $v_{n+1},\dots,v_{n+l}$ the new vertices with $l\geq2$. 
If $(\underline{X},\underline{A})$ 
is a family of  pairs of CW-complexes such that for each $i$ the inclusion $A_i\longhookrightarrow X_i$ is null-homotopic or 
$A_i=\emptyset$  and $X_j$ is connected for all $j>n$, then 
$$\aste{Z}_K(\underline{X},\underline{A})\simeq\bigvee_{\sigma\in L}\left(\bigvee_{l-1}\Sigma\mathrm{lk}_L(\sigma)*\aste{D}(\sigma)\vee\bigvee_{j=1}^l\mathrm{lk}_L(\sigma)*\aste{D}(\sigma\cup\{v_{n+j}\})\right)$$
$$\simeq\bigvee_{\sigma\in K}\mathrm{lk}(\sigma)*\aste{D}(\sigma).$$
\end{cor}
\begin{proof}
For $l=2$, $\aste{Z}_K(\underline{X},\underline{A})=Y_1\cup Y_2$ where $Y_i$ is the polyhedral join of the cone of $L$ with apex vertex 
$v_{n+i}$ and the same vertex set as $K$. Then $Y_1\cap Y_2=\aste{Z}_{\hat{L}}(\underline{X},\underline{A})$, where $\hat{L}$ comes 
from $L$ by taking $v_{n+1},v_{n+2}$ as phantom vertices. Then $Y_1\cap Y_2\longhookrightarrow Y_i$ is null-homotopic and by Lemma 
\ref{homocolimpegado} we have that $\aste{Z}_K(\underline{X},\underline{A})\simeq Y_1\vee Y_2\vee\Sigma Y_1\cap Y_2$ and by Corollaries 
\ref{corpolyjoinnullpairs} and \ref{corpoljoincone} we obtain the result. 

As before, for $l\geq3$ we have that $\aste{Z}_K(\underline{X},\underline{A})=Y_1\cup\cdots\cup Y_l$ and assuming that the result its 
true for $l-1$, we have that $\left(Y_1\cup\dots\cup Y_{l-1}\right)\cap Y_l=\hat{L}$ define as before and that the inclusion are 
null-homotopic, by Lemma \ref{homocolimpegado} and Corollary \ref{corpoljoincone} the result is true.

To see the last equivalence, we take $\sigma$ in $K$. If $\sigma$ is in $L$, then 
$$\mathrm{lk}_K(\sigma)=\{\{v_{n+1}\}\}*\mathrm{lk}_L(\sigma)\cup\cdots\cup\{\{v_{n+l}\}\}*\mathrm{lk}_L(\sigma)\simeq\bigvee_{l-1}\Sigma\mathrm{lk}_L(\sigma).$$
If $\sigma$ is not in $L$, then $v_{n+j}$ is a vertex of $\sigma$ for only one $1\leq j\leq l$ and 
$\mathrm{lk}_K(\sigma)=\mathrm{lk}_L(\sigma-\{v_{n+j}\})$.
\end{proof}

Now we proceed with the calculation of the homotopy type of a specific family of polyhedral joins that will be 
needed in the last section.

\begin{prop}\label{propskdelta}
For $d\leq n$,
$$\aste{Z}_{_{\mathrm{sk}_d\Delta^n}}\left(\bigvee_{r-1}\mathbb{S}^0,\emptyset\right)\simeq\bigvee_{f_d(r,n)}\mathbb{S}^d,$$
where
$$f_d(r,n)=\sum_{i=0}^{d+1}(-1)^{d+1-i}\binom{n+1}{i}r^i.$$
\end{prop}
\begin{proof}
We will set $X=\displaystyle\bigvee_{r-1}\mathbb{S}^0$. Now, for $d=n$, 
$$\aste{Z}_{_{\mathrm{sk}_d\Delta^n}}\left(X,\emptyset\right)=\aste{Z}_{_{\Delta^n}}\left(X,\emptyset\right)=\bigast_{i=1}^{n+1}\left(\bigvee_{r-1}\mathbb{S}^0\right)\simeq\bigvee_{(r-1)^{n+1}}\mathbb{S}^n.$$
We will use induction on $d$ and for each $d$, induction on $n$. For $d=1$, $\aste{Z}_{_{\mathrm{sk}_1\Delta^n}}\left(X,\emptyset\right)$
is the complete $(n+1)$-partite graph with $r$ vertices in each partition. Therefore:
$$\aste{Z}_{_{\mathrm{sk}_d\Delta^n}}\left(X,\emptyset\right)\simeq\bigvee_{\binom{n+1}{2}r^2-(n+1)r+1}\mathbb{S}^1.$$
Now, assume it is true for $d-1$ and any $n$ and also for $(d,n-1)$; and consider the case $(d,n)$. By case analysis on the first vertex of $\Delta^n$, we obtain:
$$\aste{Z}_{_{\mathrm{sk}_d\Delta^n}}\left(X,\emptyset\right)=\left[\left(\bigvee_{r-1}\mathbb{S}^0\right)*\aste{Z}_{_{\mathrm{sk}_{d-1}\Delta^{n-1}}}\left(X,\emptyset\right)\right]\bigcup\aste{Z}_{_{\mathrm{sk}_d\Delta^{n-1}}}\left(X,\emptyset\right).$$
Since the intersection of those two subcomplexes is $\aste{Z}_{_{\mathrm{sk}_{d-1}\Delta^{k-1}}}\left(X,\emptyset\right)$, we can conclude that $\aste{Z}_{_{\mathrm{sk}_d\Delta^n}}\left(X,\emptyset\right)$ is homotopy equivalent to the homotopy pushout of 
\begin{equation*}
    \xymatrix{
    \displaystyle\bigvee_{(r-1)f_d(r,n)}\mathbb{S}^d \ar@{<-^)}[r]& \displaystyle\bigvee_{f_d(r,n)}\mathbb{S}^{d-1} \ar@{^(->}[r]& \displaystyle\bigvee_{f_{d+1}(r,n)}\mathbb{S}^d.
    }
\end{equation*}
Both inclusions in that diagram must be null-homotopic, so Lemma \ref{homocolimpegado} applies, and we obtain the desired homotopy type: a wedge of $f_d(r,n)$
copies of $\mathbb{S}^d$, where
$$f_{d}(r,n):=r f_{d-1}(r,n-1)+f_d(r,n-1).$$

Now we need only prove the stated formula for $f_d(r,n)$. For $d=1$ or $d=n$ we know it is true. Assume the formula is true for $d-1$ and any $n$ and also for $(d,n-1)$; and consider the case of $(d,n)$. Now,
$$f_d(r,n)=\sum_{i=0}^{d}(-1)^{d-i}\binom{n}{i}r^{i+1}+\sum_{i=0}^{d+1}(-1)^{d+1-i}\binom{n}{i}r^i$$
Reindexing the first sum from $i=1$ to $d+1$, and using a standard identity for binomial coefficients, we obtain the desired formula.
\end{proof}

\section{Polyhedral joins over independence complexes of graphs}
In this section we give some homotopy type decompositions for polyhedral joins over graph complexes, theses 
calculations will give us  the homotopy type of the independence complexes of lexicographic product for these families
with any graph which is one of the goals in the next section. We wanna point out that all of the computations for specific 
families of graphs will tell us that the formula of Corollary \ref{corpolyjoinnullpairs} is valid without suspensions.
As first example we have the following proposition. 
\begin{prop}
Let $M_q$ be the graph form by $q$ disjoint edges, then 
$$\aste{Z}_{\mathcal{F}_0(M_q)}(\underline{X},\underline{A})\simeq\bigvee_{\sigma\in\mathcal{F}_0(M_q)}\mathrm{lk}(\sigma)*\aste{D}(\sigma)$$
\end{prop}
\begin{proof}
For $q=1$ the result is clear. For $q\geq2$, is a direct consequence of Corollary \ref{corsuspcompl}.
\end{proof}

In \citep{Przytycki_2017} it was proved that $\mathcal{F}_0(G)\simeq\mathcal{F}_0(G-v)\vee\Sigma\mathcal{F}_0(G-N_G[v])$
for a graph $G$  with vertices $u,v$ such that $N_G[u]\subseteq N_G[v]$. The following proposition is a generalization of this 
result, which is obtained in the case $G\circ K_1$.
\begin{prop}\label{proppoljoincloseneigh}
Let $G$ be a graph with vertices $u,v$ such that $N_G[u]\subseteq N_G[v]$. Let $(\underline{X},\underline{A})$ be a family 
of CW pairs such that for every vertex $w$ of $G$ we have that either $A_w$ is empty or $A_w\longhookrightarrow X_w$ is 
null-homotopic. Taking the complexes $J=\mathcal{F}_0(G-N(v))$, $K=\mathcal{F}_0(G-N[v])$, $L=\mathcal{F}_0(G-v)$ which have $V(G)$ as 
vertex set, we have that:
\begin{enumerate}
    \item If $K\neq\{\emptyset\}$ or $A_w\neq\emptyset$ for some vertex $w$, then 
    $$\aste{Z}_{\mathcal{F}_0(G)}(\underline{X},\underline{A})\simeq\aste{Z}_{J}(\underline{X},\underline{A})\vee\Sigma \aste{Z}_{K}(\underline{X},\underline{A})\vee \aste{Z}_{L}(\underline{X},\underline{A})$$
    \item  If $K=\{\emptyset\}$ and the family is $(\underline{X},\underline{\emptyset})$, then 
    $$\aste{Z}_{\mathcal{F}_0(G)}(\underline{X},\underline{A})\simeq X_v\sqcup\aste{Z}_{L}(\underline{X},\underline{A})$$
\end{enumerate}
\end{prop}
\begin{proof}
We take the complexes $J=\mathcal{F}_0(G-N(v))$, $K=\mathcal{F}_0(G-N[v])$, $L=\mathcal{F}_0(G-v)$ with $V(G)$ as their
vertex set. Then
$$\aste{Z}_{\mathcal{F}_0(G)}(\underline{X},\underline{A})=\aste{Z}_{J}(\underline{X},\underline{A})\cup\aste{Z}_{L}(\underline{X},\underline{A})$$
and 
$$\aste{Z}_{J}(\underline{X},\underline{A})\cap\aste{Z}_{L}(\underline{X},\underline{A})=\aste{Z}_{K}(\underline{X},\underline{A})$$
If $K=\{\emptyset\}$ and the family is $(\underline{X},\underline{\emptyset})$, the result is clear because $N_G[v]=V(G)$ and 
$\aste{Z}_{K}(\underline{X},\underline{A})=\emptyset$. Assume $K\neq\{\emptyset\}$ or $A_w\neq\emptyset$ for some vertex $w$. 
Now, $\aste{Z}_{J}(\underline{X},\underline{A})$ is a space that looks like $A_u*X_v*Y$ and $\aste{Z}_{K}(\underline{X},\underline{A})$ 
looks like $A_u*A_v*Y$, therefore the inclusion is null-homotopic. Also, we can factor the other inclusion as  
$$\aste{Z}_{K}(\underline{X},\underline{A})\longhookrightarrow X_u*A_v*Y\longhookrightarrow\aste{Z}_{L}(\underline{X},\underline{A}),$$ 
thus this inclusion also is null-homotopic. The result follows by Lemma \ref{homocolimpegado}.
\end{proof}

We will apply last proposition together with the next lemma.

\begin{lem}\label{lemlkclosnghb}
Let $G$ be a graph with vertices $u,v$ such that $N_G[u]\subseteq N_G[v]$.
If $\sigma$ is a non-empty simplex of $\mathcal{F}_0(G)$ such that $\mathrm{lk}_{\mathcal{F}_0(G)}(\sigma)$ is not contractible, then:
\begin{itemize}
    \item If $v$ is in $\sigma$, then: $\sigma$ is a simplex of $\mathcal{F}_0(G-N_G(v))$, $\sigma$ is not a simplex of $\mathcal{F}_0(G-v)$
    $\mathrm{lk}_{\mathcal{F}_0(G)}(\sigma)\cong\mathrm{lk}_{\mathcal{F}_0(G-N_G(v))}(\sigma)$.
    \item If $v$ is not in $\sigma$ and $\{v\}\cup\sigma$ is not a simplex, then: $\sigma$ is not a simplex of $\mathcal{F}_0(G-N_G(v))$, $\sigma$ is a simplex of $\mathcal{F}_0(G-v)$ and
    $\mathrm{lk}_{\mathcal{F}_0(G)}(\sigma)\cong\mathrm{lk}_{\mathcal{F}_0(G-v)}(\sigma)$.
    \item If $v$ is not in $\sigma$ and $\{v\}\cup\sigma$ is a simplex, then: $\sigma$ is a simplex of $\mathcal{F}_0(G-N_G[v])$, $\mathrm{lk}_{\mathcal{F}_0(G-N_G(v))}(\sigma)$ is contractible and
    $\mathrm{lk}_{\mathcal{F}_0(G)}(\sigma)\simeq\mathrm{lk}_{\mathcal{F}_0(G-v)}(\sigma)\vee\Sigma\mathrm{lk}_{\mathcal{F}_0(G-N_G[v])}(\sigma)$.
\end{itemize}
\end{lem}
\begin{proof}
We know that 
$$\mathcal{F}_0(G)=\mathcal{F}_0(G-N_G(v))\cup\mathcal{F}_0(G-v)$$
Take $\sigma$ a non-empty simplex of $\mathcal{F}_0(G)$ such that $\mathrm{lk}_{\mathcal{F}_0(G)}(\sigma)$ is not contractible. If $v$ is in $\sigma$, then $\sigma$ is not in $\mathcal{F}_0(G-v)$. 
If $\tau$ is a simplex of $\mathcal{F}_0(G-v)$ such that $\tau$ is in $\mathrm{lk}_{\mathcal{F}_0(G)}(\sigma)$, then 
$\tau\cap N_G(v)=\emptyset$. Thus
$\mathrm{lk}_{\mathcal{F}_0(G)}(\sigma)\cong\mathrm{lk}_{\mathcal{F}_0(G)-N_G(v)}(\sigma)$. If $v$ is not in $\sigma$ and $\{v\}\cup\sigma$ is not a simplex, then $N_G(v)\cap\sigma$ is not 
empty. Therefore $\sigma$ is not in $\mathcal{F}_0(G-N_G(v))$  
and $\mathcal{F}_0(G-N_G(v))\cap\mathrm{lk}_{\mathcal{F}_0(G)}(\sigma)\subseteq\mathcal{F}_0(G-v)(\sigma)$. Thus
$\mathrm{lk}_{\mathcal{F}_0(G)}(\sigma)\cong\mathrm{lk}_{\mathcal{F}_0(G-v)}(\sigma)$. Lastly we assume that $v$ is not in $\sigma$ and 
$\{v\}\cup\sigma$ is a simplex. Then $\sigma$ is in $\mathcal{F}_0(G-N_G(v))\cap\mathcal{F}_0(G-v)$ and 
$$\mathrm{lk}_{\mathcal{F}_0(G)}(\sigma)=\colim\left(\mathrm{lk}_{\mathcal{F}_0(G-v)}(\sigma)\longhookleftarrow\mathrm{lk}_{\mathcal{F}_0(G-N_G[v])}(\sigma)\longhookrightarrow\mathrm{lk}_{\mathcal{F}_0(G-N_G(v))}(\sigma)\right).$$
The left inclusion can be factored as 
$$\mathrm{lk}_{\mathcal{F}_0(G-N_G[v])}(\sigma)\longhookrightarrow\{u\}*\mathrm{lk}_{\mathcal{F}_0(G-N_G[v])}(\sigma)\longhookrightarrow\mathrm{lk}_{\mathcal{F}_0(G-v)}(\sigma),$$
thus it is null-homotopic. Now $\mathrm{lk}_{\mathcal{F}_0(G-N_G(v))}(\sigma)$ is contractible beause is a cone with apex vertex $v$, then 
by Lemma \ref{homocolimpegado} the link has the desired homotopy type.
\end{proof}

\begin{theorem}\label{poljointrees}
Let $T$ be a tree and $(\underline{X},\underline{A})$ be a family of CW pairs such that for every vertex $w$ of $T$ 
we have that either $A_w$ is empty or $A_w\longhookrightarrow X_w$ is null-homotopic. Then
\begin{enumerate}
    \item If $T$ is a star with $\underline{A}=\underline{\emptyset}$ and $v$ is the central vertex, then 
    $$\aste{Z}_{\mathcal{F}_0(T)}(\underline{X},\underline{A})=X_v\sqcup\bigast_{u\in N_T(v)}X_u$$
    \item If $T$ is not a star or there is a vertex $w$ with $A_w\neq\emptyset$, then
    $$\aste{Z}_{\mathcal{F}_0(T)}(\underline{X},\underline{A})\simeq\bigvee_{\sigma\in\mathcal{F}_0(T)}\mathrm{lk}(\sigma)*\aste{D}(\sigma)$$
\end{enumerate}
\end{theorem}
\begin{proof}
If $T=K_1,n$ with $v$ its central vertex, then 
$$\aste{Z}_{\mathcal{F}_0(T)}(\underline{X},\underline{A})\simeq\hocolim\left(X_v*\bigast_{u\neq v}A_u\longhookleftarrow\bigast_{u\in V(T)}A_u\longhookrightarrow A_v*\bigast_{u\neq v}X_u\right)$$
The result is clear if $\underline{A}=\underline{\emptyset}$. Assume there is a vertex $w$ such that $A_w\neq\emptyset$. Thus 
$$\aste{Z}_{\mathcal{F}_0(T)}(\underline{X},\underline{A})\simeq X_v*\bigast_{u\neq v}A_u\vee\Sigma\left(\bigast_{u\in V(T)}A_u\right)\vee A_v*\bigast_{u\neq v}X_u$$
Now, in $\mathcal{F}_0(T)$ the only simplices with a non-contractible link are $\emptyset,\{v\}$ and $N_G(v)$, and their respective links have the homotopy type of $\mathbb{S}^0,\emptyset$ and $\emptyset$
respectively. From all these we obtain the result.

Notice that if a tree $T$ is not a star, then $T$ has at least $4$ vertices. Taking $\underline{A}\neq\underline{\emptyset}$, we assume the result is true for any tree with at most 
$n$ vertices and we take $T$ to be a tree with $n+1$ vertices, which we can assume is not a star. 
Let $u$ be a leaf and $v$ be its only neighbor. By Proposition \ref{proppoljoincloseneigh}:
$$\aste{Z}_{\mathcal{F}_0(T)}(\underline{X},\underline{A})\simeq \aste{Z}_{J}(\underline{X},\underline{A})\vee\Sigma \aste{Z}_{K}(\underline{X},\underline{A})\vee \aste{Z}_{L}(\underline{X},\underline{A})$$
where $J=\mathcal{F}_0(T-N(v))$, $K=\mathcal{F}_0(T-N[v])$, $L=\mathcal{F}_0(T-v)$ and $V(T)$ is their vertex set. By induction using Lemma \ref{poljoinovjoin}, Corollary \ref{corpolyjoinnullpairs}, 
Corollary \ref{corpoljoincone} and Lemma \ref{lemlkclosnghb} we only need to check the link of the simplex $\emptyset$. Now, 
$$\mathrm{lk}(\emptyset)=\mathcal{F}_0(T)\simeq L\vee\Sigma K.$$
\end{proof}

In the next theorem we calculate the homotopy type of polyhedral joins over the independence complex of a power of a path. 
The $r$-th power of $P_n$ is the graph $P_n^r$ which has the same vertex set as $P_n$ and two distinct vertices are adjacent if their distance is at most $r$, \textit{i.e.} the only  
path between them in $P_n$ has at most $r$ vertices. We will take $\{v_1,\dots,v_n\}$ as the vertex set of $P_n$, taking the obvious edge set we have that the distance between $v_i$ and $v_j$ is $|i-j|$.
Now we want to point out that because the independence number of $P_n^r$ is
$$\alpha(P_n^r)=\left\lceil\frac{n}{r+1}\right\rceil,$$
for $n\leq2r+2$ the complex $\mathcal{F}_0(P_n^r)$ has dimension at most $1$, more precisely for $r+2\leq n\leq2r+2$
$$\mathcal{F}_0(P_n^r)\cong H_{n,r}\sqcup lK_1$$
where $l=|\{v_j:\;n-r\leq j\leq r+1\}|$ and, $H_{n,r}$ is the bipartite graph with vertex set $\{v_1,\dots,v_{n-r-1}\}\cup\{v_{r+2},\dots,v_{n}\}$ 
and $v_iv_j$ is an edge if $|j-i|\geq r+1$.
\begin{theorem}\label{poljoinpathpow}
For $P_n^r$ let $(\underline{X},\underline{A})$ be a family of CW pairs such that for every vertex $w$ of $T$ 
we have that either $A_w$ is empty or $A_w\longhookrightarrow X_w$ is null-homotopic. Then
\begin{enumerate}
    \item For $n\leq r+1$:
    \begin{enumerate}
        \item If $\underline{A}=\underline{\emptyset}$, then
        $\displaystyle\aste{Z}_{\mathcal{F}_0(P_n^r)}(\underline{X},\underline{\emptyset})=\bigsqcup_{i=1}^nX_{v_i}.$
        \item If there is a $j\in\underline{n}$ such that $A_j\neq\emptyset$, then 
        $$\displaystyle\aste{Z}_{\mathcal{F}_0(P_n^r)}(\underline{X},\underline{A})\simeq\bigvee_{i=1}^nX_{v_i}\vee\bigvee_{n-1}\Sigma\left(\bigast_{u\in V(P_n)}A_u\right).$$
    \end{enumerate}
    \item For $r+2\leq n\leq2r+2$
    \begin{enumerate}
        \item If $\underline{A}=\underline{\emptyset}$, then 
        $\displaystyle\aste{Z}_{\mathcal{F}_0(P_n^r)}(\underline{X},\underline{\emptyset})=\displaystyle\aste{Z}_{H_{n,r}}(\underline{X},\underline{\emptyset})\sqcup\bigsqcup_{n-r\leq i\leq r+1}X_{v_i}.$
        \item If there is a $j\in\underline{n}$ such that $A_j\neq\emptyset$, then
        $$\aste{Z}_{\mathcal{F}_0(P_n^r)}(\underline{X},\underline{A})\simeq\displaystyle\aste{Z}_{H_{n,r}}(\underline{X},\underline{A})\vee\bigvee_{n-r\leq i\leq r+1}X_{v_i}\vee\bigvee_{n-r\leq i\leq r}\Sigma\left(\bigast_{u\in V(P_n)}A_u\right).$$
    \end{enumerate}
    \item For $n\geq2r+3$ 
    $$\aste{Z}_{\mathcal{F}_0(P_n^r)}(\underline{X},\underline{A})\simeq\bigvee_{\sigma\in\mathcal{F}_0(P_n^r)}\mathrm{lk}(\sigma)*\aste{D}(\sigma)$$
\end{enumerate}
\end{theorem}
\begin{proof}
For $r=1$, $P_n^1=P_n$ is a particular case of Theorem \ref{poljointrees}.
Assume $r\geq2$. For $n\leq2r+2$ the result is clear by the observation previous to the Theorem and doing the respective calculation 
using Lemma \ref{homocolimpegado}. We assume that $n\geq2r+3$.

We apply Proposition \ref{proppoljoincloseneigh} to the vertices $v_1,v_{r+1}$, then to the vertices $v_1,v_{r}$ in the graph 
$P_n^r-v_{r+1}$ and so on until we apply it to the the vertices $v_1,v_2$ in the graph $P_n^r-v_{r+1}-\cdots-v_3$. 
In all the steps we always take $V(P_n^r)$ as the vertex set of all of these complexes. Then, we have the following
$$\aste{Z}_{\mathcal{F}_0(P_n^r)}(\underline{X},\underline{A})\simeq\bigvee_{i=2}^{r+1}\Sigma\aste{Z}_{K_i}(\underline{X},\underline{A})\vee\bigvee_{j=1}^{r+1}\aste{Z}_{L_j}(\underline{X},\underline{A})$$
where $K_i=\mathcal{F}_0(P_n^r-N[v_i])$, $L_j=\mathcal{F}_0(P_n^r-N(v_i))$ and all have the same vertex set as $P_n^r$.
By Corollaries \ref{corpolyjoinnullpairs} and \ref{corpoljoincone} we have that 
$$\aste{Z}_{\mathcal{F}_0(P_n^r)}(\underline{X},\underline{A})\simeq\bigvee_{i=2}^{r+1}\bigvee_{\sigma\in K_i}\Sigma\mathrm{lk}_{K_i}(\sigma)*\aste{D}(\sigma)\vee\bigvee_{j=1}^{r+1}\bigvee_{\sigma\in K_j}\mathrm{lk}_{K_j}(\sigma)*\aste{D}(\sigma\cup\{v_j\})$$
By Lemma \ref{lemlkclosnghb} we only need to check the link of $\emptyset$. Now 
$$\mathrm{lk}(\emptyset)=\mathcal{F}_0(P_n^r)\simeq\bigvee_{i=2}^{r+1}\Sigma K_i.$$
\end{proof}

Given a graph $G$ and $\sigma$ an independent set of $G$, the \textit{star cluster} of $\sigma$ is the complex 
$$SC(\sigma)=\bigcup_{v\in\sigma}\mathcal{F}_0(G-N_G(v))$$
In \citep{barmak} it is shown that $\mathcal{F}_0(G)\simeq\Sigma\mathcal{F}_0(G-N_G(v))\cap SC(N_G(v))$ if $N_G(v)$ is an independent 
set. The next lemma is a generalization of this result.

\begin{lem}\label{lemstrclstr}
Let $G$ be a graph with a vertex $v$ such that it is not contained in any triangle. 
Take $(\underline{X},\underline{A})$ a family of CW pairs such that for every vertex $w$ of $G$ we have that $A_w$ is empty or that 
$A_w\longhookrightarrow X_w$ is null-homotopic. Take 
$\iota\colon\aste{Z}_{L}(\underline{X},\underline{A})\longhookrightarrow\aste{Z}_{K}(\underline{X},\underline{A})$, where
$J=\mathcal{F}_0(G-N_G(v))$, $K=SC(N_G(v))$, $L=J\cap K$ and all of these complexes have $V(G)$ as vertex set. Then
$$\aste{Z}_{\mathcal{F}_0(G)}(\underline{X},\underline{A})\simeq\aste{Z}_{J}(\underline{X},\underline{A})\vee\mathrm{hocofiber}(\iota).$$
If $v$ has degree $2$, then
$$\aste{Z}_{\mathcal{F}_0(G)}(\underline{X},\underline{A})\simeq\aste{Z}_{J}(\underline{X},\underline{A})\vee\Sigma\aste{Z}_{L}(\underline{X},\underline{A})\vee\aste{Z}_{K}(\underline{X},\underline{A})$$
\end{lem}
\begin{proof}
We have that 
$$\aste{Z}_{\mathcal{F}_0(G)}(\underline{X},\underline{A})=\aste{Z}_{J}(\underline{X},\underline{A})\cup\aste{Z}_{K}(\underline{X},\underline{A})$$
Now, $\aste{Z}_{L}(\underline{X},\underline{A})\longhookrightarrow\aste{Z}_{J}(\underline{X},\underline{A})$ is null-homotopic 
because $\aste{Z}_{L}(\underline{X},\underline{A})$ is a space which looks like $A_v*Y$ and $\aste{Z}_{J}(\underline{X},\underline{A})$
like $X_v*Y$ where $A_v$ is either empty or $A_v\longhookrightarrow X_v$ is null-homotopic. Therefore 
$$\aste{Z}_{\mathcal{F}_0(G)}(\underline{X},\underline{A})\simeq\aste{Z}_{J}(\underline{X},\underline{A})\vee\mathrm{hocofiber}(\iota).$$ 
Assume $N_G(v)=\{v_1,v_2\}$. Taking 
$\mathrm{st}(v_1),\mathrm{st}(v_2),\mathrm{st}(v)\cap\mathrm{st}(v_1),\mathrm{st}(v)\cap\mathrm{st}(v_2)$ as complexes 
with vertex set $V(G)$, we have that 
$$\aste{Z}_{K}(\underline{X},\underline{A})=\aste{Z}_{\mathrm{st}(v_1)}(\underline{X},\underline{A})\cup\aste{Z}_{\mathrm{st}(v_2)}(\underline{X},\underline{A})=W_1\cup W_2$$
$$\aste{Z}_{L}(\underline{L},\underline{A})=\aste{Z}_{\mathrm{st}(v)\cap\mathrm{st}(v_1)}(\underline{X},\underline{A})\cup\aste{Z}_{\mathrm{st}(v)\cap\mathrm{st}(v_2)}(\underline{X},\underline{A})=Y_1\cup Y_2.$$
We take $W_0=W_1\cap W_2$, $Y_0=Y_1\cap Y_2$. From the following diagram
\begin{equation*}
\xymatrix{
\ast \ar@{<-}[rr] \ar@{<-}[d]&& Y_1  \ar@{^(->}[rr]_{\theta_1}&& W_1\ar@{<-^)}[d]\\
\ast \ar@{<-}[rr] \ar@{->}[d]&& Y_0 \ar@{^(->}[rr]_{\theta_0} \ar@{^(->}[u] \ar@{_(->}[d]&& W_0 \ar@{^(->}[d]\\
\ast \ar@{<-}[rr] && Y_2 \ar@{^(->}[rr]_{\theta_2} && W_2
}
\end{equation*}
we have that
$$\hocolim\left(I\longleftarrow\hocolim\mathcal{T}\longhookrightarrow\hocolim\mathcal{S}\right)\cong$$ 
$$\hocolim\left(\mathrm{hocofiber}(\theta_1)\longhookleftarrow\mathrm{hocofiber}(\theta_0)\longhookrightarrow\mathrm{hocofiber}(\theta_2)\right)$$
where $\mathcal{T}$ is the middle column and $\mathcal{S}$ the right-most one \citep[see Theorem 3.7.15]{cubicalhomotopy}. 

Taking 
$\theta:\hocolim\mathcal{T}\longhookrightarrow\hocolim\mathcal{S}$, we have that 
$\mathrm{hocofiber}(\theta)\simeq\hocolim\left(I\longleftarrow\hocolim\mathcal{T}\longhookrightarrow\hocolim\mathcal{S}\right)$. 
Now 
$$\mathrm{hocofiber}(\iota)\cong\colim\left(\mathrm{hocofiber}(\theta_1)\longhookleftarrow\mathrm{hocofiber}(\theta_0)\longhookrightarrow\mathrm{hocofiber}(\theta_2)\right)$$
$$\simeq\hocolim\left(\mathrm{hocofiber}(\theta_1)\longhookleftarrow\mathrm{hocofiber}(\theta_0)\longhookrightarrow\mathrm{hocofiber}(\theta_2)\right),$$
thus $\mathrm{hocofiber}(\iota)\simeq\mathrm{hocofiber}(\theta)$.
Notice that the spaces of the diagram can be seen as:
$$Y_1=A_{v_1}\ast B_1,\;\;Y_2=A_{v_2}\ast B_2,\;\;Y_0=A_{v_1}\ast A_{v_2}\ast B_0,$$
$$W_1=X_{v_1}\ast B_1',\;\;W_2=X_{v_2}\ast B_2',\;\;W_0=X_{v_1}\ast X_{v_2}\ast B_0.$$
where 
$$B_i'=\bigast_{w\in N_G(v_i)}A_w*\aste{Z}_{\mathcal{F}_0(G-N_G[v_i])}(\underline{X},\underline{A}),\;\;B_i=\bigast_{w\in N_G(v_i)}A_w*A_{v_{i+1}}*\aste{Z}_{\mathcal{F}_0(G-N_G[v_i])}(\underline{X},\underline{A})$$
$B_0'=B_1'\cap B_2'$, $B_0=B_1\cap B_2$ and the vertex set of $\mathcal{F}_0(G-N_G[v_i])$ is $V(G)-N_G[v_i]$.
We will show that $\theta$ is null-homotopic. There are three possible cases.
If $A_{v_1}=A_{v_2}=\emptyset$, we take $x_i$ a vertex of $X_{v_i}$ and define the spaces $Z_i=\{x_i\}*B_i'$ and 
$Z_0=\{\{v_1\},\{v_2\}\}*B_0'$. We factorize $\theta$ through the homotopy push-out 
$$\hocolim\left(Z_1\longhookleftarrow Z_0\longhookrightarrow Z_2\right)$$
taking the diagram 
\begin{equation*}\tag{$f$}\label{facttheta}
\xymatrix{
Y_1 \ar@{<-^)}[rr] \ar@{^(->}[d]&& Y_0 \ar@{^(->}[rr] \ar@{^(->}[d]&& Y_2 \ar@{^(->}[d]\\
Z_1 \ar@{<-^)}[rr] \ar@{^(->}[d]&& Z_0 \ar@{^(->}[rr] \ar@{^(->}[d]&& Z_2 \ar@{^(->}[d]\\
W_1 \ar@{<-^)}[rr] && W_0 \ar@{^(->}[rr]&& W_2
}
\end{equation*}
Notice that $\hocolim\left(Z_1\longhookleftarrow Z_0\longhookrightarrow Z_2\right)\simeq*$. Therefore $\theta$ is null-homotopic.

Assume $A_{v_1}\neq\emptyset\neq A_{v_2}$. By hypothesis $A_{v_i}\longhookrightarrow X_{v_i}$ is null-homotopic. We take the following 
commutative diagram
\begin{equation*}
\xymatrix{
A_{v_i} \ar@{^(->}[rr]^{j_0} \ar@{^(->}[dd]_{\iota_i}& & A_{v_i}\times I \ar@{->}[dd]^{h_{v_i}} \ar@{->}[ld]_{\iota_i\times1}\\
 & X_{v_i}\times I \ar@{-->}[dr]^{H_{v_i}}& \\
 X_{v_i} \ar@{^(->}[ru]^{j_0} \ar@{->}[rr]_{1}& & X_{v_i}
}
\end{equation*}
where $h_{v_i}$ is the homotopy between $A_{v_i}\longhookrightarrow X_{v_i}$ and the constant map $c(x)=x_i$, where 
$x_i$ is a vertex of $X_{v_i}$. We get a map $H_{v_i}:X_{v_i}\times I\longrightarrow X_{v_i}$ 
such that $H_{v_i}(x,0)=x$ and $=H_{v_i}(a,1)=x_i$ for all $a$ in $A_{v_i}$. 
We define $g_i:X_{v_i}\longrightarrow X_{v_i}$  as $H_{v_i}(x,1)$. Thus $g_i$ is homotopic to 
the identity and ${g_i}_{|_{A_{v_i}}}$ is a constant map for $i=1,2$. With these maps we can construct 
homotopy equivalences $f_1=g_1*1_{B_n'}$, $f_2=g_2*1_{B_2'}$ and $f_0=g_1*g_2*1_{B_0'}$ such that the following diagram is commutative up to homotopy

\begin{equation*}
\xymatrix{
\mathcal{R}_1=\mathcal{S}&W_1 \ar@{<-^)}[rr] \ar@{->}[d]^{f_1}&& W_0 \ar@{^(->}[rr] \ar@{->}[d]_{f_0}&& W_2 \ar@{->}[d]_{f_2}\\
\mathcal{R}_2&W_1 \ar@{<-}[rr]_{g_1*\iota_1} && W_0 \ar@{->}[rr]_{g_2*\iota_2} && W_2
}
\end{equation*}
with homotopies $H_1$ and $H_2$ for the first and second square respectively. Then there is a homotopy equivalence $\Lambda:\hocolim\mathcal{R}_1\longrightarrow\hocolim\mathcal{R}_2$, where:
\begin{itemize}
    \item For $x\in W_i$ with $i\neq0$, $\Lambda(x)=f_i(x)$.
    \item For $(x,t)\in W_0\times I$:
    \begin{enumerate}
        \item If $0\leq t\leq\frac{1}{3}$, then $\Lambda((x,t))=H_1(x,1-3t)$. 
        \item If $\frac{1}{3}\leq t\leq\frac{2}{3}$, then $\Lambda((x,t))=(f_0(x),3t-1)$.
        \item If $\frac{2}{3}\leq t\leq1$, then $\Lambda((x,t))=H_2(x,3t-2)$.
    \end{enumerate}
\end{itemize}
Notice that $\mathrm{Im}(\Lambda\circ\theta)\simeq*$, therefore $\theta$ is null-homotopic. 

Lastly assume $A_{v_1}\neq\emptyset=A_{v_2}$. Using a mix of the ideas of the previous two cases we will show that $\theta$ is null-homotopic. 
We take $x_2$ a vertex of $X_{v_2}$ and $g_1$ as before, we define 
$Z_0=\{x_2\}*B_0'$, $Z_2=\{x_2\}*B_2'$ and take $Z_1=W_1$.
We define $f_1$ as before and $f_0=g_1*1_{\{v_2\}}*1_{B_0'}$ . We can factorize $\theta$ as in diagram (\ref{facttheta}). We take the following homotopy commutative diagram
\begin{equation*}
\xymatrix{
W_1 \ar@{<-^)}[rr] \ar@{->}[d]^{f_1}&& W_0 \ar@{^(->}[rr] \ar@{->}[d]_{f_0}&& W_2 \ar@{->}[d]_{1_{W_2}}\\
W_1 \ar@{<-}[rr]_{g_1*\iota_1} && W_0 \ar@{^(->}[rr] && W_2
}
\end{equation*}
We can construct $\Lambda$ as before and we get that $\mathrm{Im}(\Lambda\circ\theta)\simeq*$, thus $\theta$ is null-homotopic.

Regardless of the case we have that $\mathrm{hocofiber}(\iota)\simeq\Sigma\hocolim\mathcal{T}\vee\hocolim\mathcal{S}$. 
\end{proof}
We will apply last lemma together with the next two lemmas.

\begin{lem}\label{lemstarclsindgrp}
Let $G$ be a graph with $v$ a vertex of degree $2$ such that  $N_G(v)=\{v_1,v_2\}$ is an independent set. 
We take the graph $H$ obtain from $G$ by 
first erasing $N_G(v_1)\cap N_G(v_2)$ from $G$, and next adding the edges $u_1u_2$ for all $u_1\in N_G(v_1)-N_G(v_2)$ and 
$u_2\in N_G(v_2)-N_G(v_1)$. Then 
$$\mathcal{F}_0(G-N_G(v_1))\cup\mathcal{F}_0(G-N_G(v_2))=\mathcal{F}_0(H)$$
$$\mathcal{F}_0(G-N_G(v))\cap\left(\mathcal{F}_0(G-N_G(v_1))\cup\mathcal{F}_0(G-N_G(v_2))\right)=\mathcal{F}_0(H-v_1-v_2)$$
\end{lem}
\begin{proof}
By construction $\mathcal{F}_0(G-N_G(v_1))\cup\mathcal{F}_0(G-N_G(v_2))\subseteq\mathcal{F}_0(H)$. Take $\sigma$ in $H$:
\begin{itemize}
    \item If $v_1$ is in $\sigma$, then $\sigma\cap N_G(v_1)=\emptyset$ and $\sigma$ is in $\mathcal{F}_0(G-N_G(v_1))$.
    \item If there is a $u$ in $\sigma\cap N_G(v_1)$, then $\sigma\cap N_G(v_2)=\emptyset$ and $\sigma$ is a simplex of 
    $\mathcal{F}_0(G-N_G(v_2))$.
    \item If $\sigma\cap N_G[v_1]=\emptyset$, there are three possibilities: $v_2$ is in $\sigma$; $\sigma\cap N_G(v_2)$ is non-empty; and 
    $\sigma\cap N_G[v_2]=\emptyset$. The first two cases are analogous to what we have done. In the last case $\sigma$ is in 
    $\mathcal{F}_0(G-N_G(v_1))\cap\mathcal{F}_0(G-N_G(v_2))$.
\end{itemize}
The second part of the statement is clear.
\end{proof}

\begin{lem}\label{lemlkstarcls}
Let $G$ be a graph with $v$ a vertex of degree $2$ such that  $N_G(v)=\{v_1,v_2\}$ is an independent set. Take 
$H$ as in the Lemma \ref{lemstarclsindgrp}. If $\sigma$ is a non-empty simplex of $\mathcal{F}_0(G)$ such that 
$\mathrm{lk}_{\mathcal{F}_0(G)}(\sigma)$ is not contractible, then:
\begin{itemize}
    \item If $v$ is in $\sigma$, then: $\sigma$ is a simplex of $\mathcal{F}_0(G-N_G(v))$, $\sigma$ is not a simplex of $\mathcal{F}_0(H)$ and 
    $\mathrm{lk}_{\mathcal{F}_0(G)}(\sigma)\cong\mathrm{lk}_{\mathcal{F}_0(G-N_G(v))}(\sigma)$.
    \item If $v$ is not in $\sigma$ and $\{v\}\cup\sigma$ is a simplex, then:  $\sigma$ is a simplex of $\mathcal{F}_0(G)-N_G(v)\cap\mathcal{F}_0(H)=\mathcal{F}_0(H-v_1-v_2)$ and
    $\mathrm{lk}_{\mathcal{F}_0(G)}(\sigma)\simeq\mathrm{lk}_{\mathcal{F}_0(H)}(\sigma)\vee\Sigma\mathrm{lk}_{\mathcal{F}_0(H-v_1-v_2)}(\sigma)$.
    \item If $v$ is not in $\sigma$ and $\{v\}\cup\sigma$ is not a simplex, then: $\sigma$ is a simplex of $\mathcal{F}_0(H)$, $\sigma$ is not a simplex of $\mathcal{F}_0(G-N_G(v))$ and
    $\mathrm{lk}_{\mathcal{F}_0(G)}(\sigma)\cong\mathrm{lk}_{\mathcal{F}_0(H)}(\sigma)$.
\end{itemize}
\end{lem}
\begin{proof}
First notice that 
$$\mathcal{F}_0(G)=\mathcal{F}_0(G-N_G(v))\cup\mathcal{F}_0(G-N_G(v_1))\cup\mathcal{F}_0(G-N_G(v_2))=\mathcal{F}_0(G-N_G(v))\cup\mathcal{F}_0(H).$$

Take $\sigma$ is a simplex of $\mathcal{F}_0(G)$ such that $\mathrm{lk}_{\mathcal{F}_0(G)}(\sigma)$ is not contractible. If $v$ is in $\sigma$, then 
$\sigma\cap N_G(v)=\emptyset$ and for any $w$ in $N_G(v)$, $\sigma\cup\{w\}$ is not a simplex. Therefore $\sigma$ is not a simplex of $\mathcal{F}_0(H)$ and
any simplex $\tau$ of $\mathcal{F}_0(H)\cap\mathrm{lk}_{\mathcal{F}_0(G)}(\sigma)$ is contained in $\mathcal{F}_0(G-N_G(v))$. Thus
$\mathrm{lk}_{\mathcal{F}_0(G)}(\sigma)\cong\mathrm{lk}_{\mathcal{F}_0(G-N_G(v))}(\sigma)$.

Assume $v$ is not in $\sigma$ and $\{v\}\cup\sigma$ is a simplex, then $\sigma\cap N_G(v)$ is empty. Because $\mathrm{lk}_{\mathcal{F}_0(G)}(\sigma)$ is not contractible, then 
at lest one of $\{v_1\}\cup\sigma$, $\{v_2\}\cup\sigma$ and $\{v_1,v_2\}\cup\sigma$ must be a simplex. Therefore $\sigma$ is in 
$\mathcal{F}_0(G-N_G(v))\cap\mathcal{F}_0(H)=\mathcal{F}_0(H-v_1-v_2).$
From this, we get that 
$$\mathrm{lk}_{\mathcal{F}_0(G)}(\sigma)=\colim\left(\mathrm{lk}_{_{\mathcal{F}_0(G-N_G(v))}}(\sigma)\longhookleftarrow\mathrm{lk}_{_{\mathcal{F}_0(H-v_1-v_2)}}(\sigma)\longhookrightarrow\mathrm{lk}_{_{\mathcal{F}_0(H)}}(\sigma)\right)$$
Now, $\mathrm{lk}_{_{\mathcal{F}_0(G-N_G(v))}}(\sigma)\simeq*$ because it is a cone. 
By construction of $H$ we have that 
$$\mathrm{lk}_{_{\mathcal{F}_0(H)}}(\sigma)=\mathrm{lk}_{_{\mathcal{F}_0(G-N_G(v_1))}}(\sigma)\cup\mathrm{lk}_{_{\mathcal{F}_0(G-N_G(v_2))}}(\sigma),$$
 and that $\mathrm{lk}_{_{\mathcal{F}_0(G-N_G(v_1))}}(\sigma)\simeq*\simeq\mathrm{lk}_{_{\mathcal{F}_0(G-N_G(v_2))}}(\sigma)$. By Lemma \ref{homocolimpegado} we have that
 $\mathrm{lk}_{_{\mathcal{F}_0(H)}}(\sigma)\simeq\Sigma\mathrm{lk}_{_{\mathcal{F}_0(H-N_G(v_1)-N_G(v_2))}}(\sigma)$ and that 
 $\mathrm{lk}_{_{\mathcal{F}_0(H-v_1-v_2)}}(\sigma)\longhookrightarrow\mathrm{lk}_{_{\mathcal{F}_0(H)}}(\sigma)$ is null-homotopic. 
 Thus the link has the homotopy type desired.

Lastly assume $v$ is not in $\sigma$ and $\{v\}\cup\sigma$ is not a simplex. Then $\sigma\cap N_G(v)\neq\emptyset$, $\sigma$ is not a simplex of 
$\mathcal{F}_0(G-N_G(v))$ and $\mathrm{lk}_{\mathcal{F}_0(G)}(\sigma)\cap\mathcal{F}_0(G-N_G(v))\subseteq\mathcal{F}_0(H)$, thus 
$\mathrm{lk}_{\mathcal{F}_0(G)}(\sigma)\cong\mathrm{lk}_{\mathcal{F}_0(H)}(\sigma)$.
\end{proof}

\begin{theorem}\label{teopoljoincycl}
Take $C_n$, with $n\geq5$, and let $(\underline{X},\underline{A})$ be a family of CW pairs such that for every vertex $w$ of $C_n$ 
we have that either $A_w$ is empty or $A_w\longhookrightarrow X_w$ is null-homotopic. Then
$$\aste{Z}_{\mathcal{F}_0(C_n)}(\underline{X},\underline{A})\simeq\bigvee_{\sigma\in\mathcal{F}_0(C_n)}\mathrm{lk}(\sigma)*\aste{D}(\sigma)$$
\end{theorem}
\begin{proof}

By Lemma \ref{lemstrclstr} 
$$\aste{Z}_{\mathcal{F}_0(C_n)}(\underline{X},\underline{A})\simeq\aste{Z}_{J}(\underline{X},\underline{A})\vee\aste{Z}_{K}(\underline{X},\underline{A})\vee\Sigma\aste{Z}_{L}(\underline{X},\underline{A})$$
where $J=\mathcal{F}_0(C_n-N(v_1))$, $K=SC(N(v_1))$, $L=J\cap K$ and the three complexes have $V(C_n)$ as vertex set. The only non-contractible complex is $L$, thus 
$\mathrm{lk}(\emptyset)\simeq\Sigma L$. By Lemma \ref{lemstarclsindgrp},
if we take $H=(C_n-v_1)+v_3v_{n-1}$, then $K=\mathcal{F}_0(H)$ and $L=\mathcal{F}_0(H-v_2-v_n)$. 
Notice that $H-v_2-v_n\cong C_{n-3}$ for $n\geq6$ and $H-v_2-v_n-\cong K_2$ for $n=5$.

For any $n\geq5$, by Corollary \ref{corpoljoincone}, we have that
$$\aste{Z}_{J}(\underline{X},\underline{A})\simeq\bigvee_{\sigma\in J}\mathrm{lk}_L(\sigma)*\aste{D}(\sigma)$$
Now, by Corollary \ref{corpolyjoinnullpairs}
$$\Sigma\aste{Z}_{L}(\underline{X},\underline{A})\simeq\bigvee_{\sigma\in L}\Sigma\mathrm{lk}_L(\sigma)*\aste{D}(\sigma)$$
Therefore, by Lemma \ref{lemlkstarcls},  if $\aste{Z}_{K}(\underline{X},\underline{A})$ has the homotopy type required, then the theorem is true. 
\begin{figure}
\centering
\begin{tikzpicture}[line cap=round,line join=round,>=triangle 45,x=1.5cm,y=1.5cm]
\clip(2.5,1.5) rectangle (6.5,3.5);
\fill[fill=black,fill opacity=0.53] (3,2.5) -- (4,2) -- (4,3) -- cycle;
\fill[fill=black,fill opacity=0.53] (4,2) -- (4,3) -- (5,3) -- cycle;
\fill[fill=black,fill opacity=0.53] (4,2) -- (5,2) -- (5,3) -- cycle;
\fill[fill=black,fill opacity=0.53] (5,2) -- (6,2.5) -- (5,3) -- cycle;
\draw (3,2.5)-- (4,2);
\draw (4,2)-- (4,3);
\draw (4,3)-- (3,2.5);
\draw (4,2)-- (4,3);
\draw (4,3)-- (5,3);
\draw (5,3)-- (4,2);
\draw (4,2)-- (5,2);
\draw (5,2)-- (5,3);
\draw (5,3)-- (4,2);
\draw (5,2)-- (6,2.5);
\draw (6,2.5)-- (5,3);
\draw (5,3)-- (5,2);
\begin{scriptsize}
\fill [color=black] (3,2.5) circle (1.5pt);
\draw [color=black] (2.8,2.5) node {$v_3$};
\fill [color=black] (4,3) circle (1.5pt);
\draw [color=black] (4,3.2) node {$v_5$};
\fill [color=black] (4,2) circle (1.5pt);
\draw [color=black] (4,1.8) node {$v_7$};
\fill [color=black] (5,3) circle (1.5pt);
\draw [color=black] (5,3.2) node {$v_2$};
\fill [color=black] (5,2) circle (1.5pt);
\draw [color=black] (5,1.8) node {$v_4$};
\fill [color=black] (6,2.5) circle (1.5pt);
\draw [color=black] (6.2,2.5) node {$v_6$};
\end{scriptsize}
\end{tikzpicture}
\caption{$K$ for $n=7$}\label{kn7}
\end{figure}
Now, we focus on $K$.
\begin{itemize}
    \item If $n=5$, then $H\cong P_4$ and from Theorem \ref{poljointrees} we have that
    $$\aste{Z}_{K}(\underline{X},\underline{A})\simeq\bigvee_{\sigma\in K}\mathrm{lk}(\sigma)*\aste{D}(\sigma)$$
    \item If $n=6$, then $K\cong P_5$ where the order of the vertices is $v_3,v_6,v_4,v_2,v_5$ and using Lemma \ref{homocolimpegado} is 
    easy to see that
    $$\aste{Z}_{K}(\underline{X},\underline{A})\simeq\bigvee_{\sigma\in K}\mathrm{lk}(\sigma)*\aste{D}(\sigma)$$
    \item If $n=7$, then $K$ the complex in Figure \ref{kn7} and it is easy to see that 
    $$\aste{Z}_{K}(\underline{X},\underline{A})\simeq\bigvee_{\sigma\in K}\mathrm{lk}(\sigma)*\aste{D}(\sigma)$$
    using Lemma \ref{homocolimpegado} and checking that the only simplexes without a contractible link are the triangles and the 
    edges $v_5v_7$, $v_2v_7$ and $v_2v_4$.
\end{itemize}

For $n\geq 8$, by Proposition \ref{proppoljoincloseneigh}, we have that
$$\aste{Z}_{K}(\underline{X},\underline{A})\simeq\aste{Z}_{\mathcal{F}_0(H-N_H(v_3))}(\underline{X},\underline{A})\vee\aste{Z}_{\mathcal{F}_0(H-v_3)}(\underline{X},\underline{A})\vee\Sigma\aste{Z}_{\mathcal{F}_0(H-N_H[v_3])}(\underline{X},\underline{A})$$
By Corollary \ref{corpoljoincone}, we have that 
$$\aste{Z}_{\mathcal{F}_0(H-N_H(v_3))}(\underline{X},\underline{A})\simeq\bigvee_{\sigma\in\mathcal{F}_0(H-N_H(v_3))}\mathrm{lk}(\sigma)*\aste{D}(\sigma),$$
$$\aste{Z}_{\mathcal{F}_0(H-v_3)}(\underline{X},\underline{A})\simeq\bigvee_{\sigma\in\mathcal{F}_0(H-v_3)}\mathrm{lk}(\sigma)*\aste{D}(\sigma)$$
and by Corollary \ref{corpolyjoinnullpairs}
$$\Sigma\aste{Z}_{\mathcal{F}_0(H-N_H[v_3])}(\underline{X},\underline{A})\simeq\bigvee_{\sigma\in\mathcal{F}_0(H-N_H[v_3])}\Sigma\mathrm{lk}(\sigma)*\aste{D}(\sigma)$$
It is clear that all the complexes $\mathcal{F}_0(H-v_3),\mathcal{F}_0(H-N_H(v_3),\mathcal{F}_0(H-N_H[v_3])$ are contractible. Now, $\mathcal{F}_0(H)\simeq\mathcal{F}_0(H-v_4-v_{n-1})$ because the only 
neighbor of $v_2$ is $v_3$ (\citep[see Lemma 3.2]{engstrom09}), and $\mathcal{F}_0(H-v_4-v_{n-1})$ is contractible as $v_n$ is an isolated vertex. By Lemma \ref{lemlkclosnghb}, 
$$\aste{Z}_{K}(\underline{X},\underline{A})\simeq\bigvee_{\sigma\in K}\mathrm{lk}_K(\sigma)*\aste{D}(\sigma)$$
\end{proof}

\section{Complexes of \texorpdfstring{$G\circ H$}{Î£}}
Remember that the lexicographic product $G\circ H$ is the graph obtained by taking a copy of $H$ for each vertex of $G$ and all the possible 
edges between two copies if the corresponding vertices are adjacent in $G$. First we will see that for the second factor only the 
homotopy type of its independence complex matters. 
\begin{theorem}
Let $H_1$ and $H_2$ be graphs such that $\mathcal{F}_0(H_1)\simeq \mathcal{F}_0(H_2)$, then $\mathcal{F}_0(G\circ H_1)\simeq \mathcal{F}_0(G\circ H_2)$.
\end{theorem}
\begin{proof}
If $\sigma_1,\dots,\sigma_k$ are the maximal simplexes of $\mathcal{F}_0(G)$, taking $G_i=G[\sigma_i]$, $X_i=\mathcal{F}_0(G_i\circ H_1)$ and $Y_i=\mathcal{F}_0(G_i\circ H_2)$, we 
have that $X_i\cong \mathcal{F}_0(H_1)^{*|\sigma_i|}$, $Y_i\cong \mathcal{F}_0(H_2)^{*|\sigma_i|}$. 

From this, $\mathcal{F}_0(G\circ H_1)=X_1\cup\cdots\cup X_k$ and $\mathcal{F}_0(G\circ H_2)=Y_1\cup\cdots\cup Y_k$. We take the 
punctured $k$-cubes 
$$\mathcal{X}(S)=\bigcap_{i\in S^c}X_i \;\;\;\;\mbox{and}\;\;\;\;\mathcal{Y}(S)=\bigcap_{i\in S^c}Y_i$$
with the inclusions as the maps. If $S^c=\{i_1,\dots,i_m\}$, we have that 
$$X_{i_1}\cap\cdots\cap X_{i_m}\cong\bigast_{i\in\sigma_{i_1}\cap\cdots\sigma_{i_m}}\mathcal{F}_0(H_1),\;\;\; Y_{i_1}\cap\cdots\cap Y_{i_m}\cong\bigast_{i\in\sigma_{i_1}\cap\cdots\sigma_{i_m}}\mathcal{F}_0(H_2).$$
If $f:\mathcal{F}_0(H_1)\longrightarrow \mathcal{F}_0(H_2)$ is a homotopy equivalence, taking 
$f_S:\mathcal{X}(S)\longrightarrow\mathcal{Y}(S)$ the corresponding induced homotopy equivalence if 
$\displaystyle\bigcap_{i\in S^c}\sigma_i\neq\emptyset$, we have that the collection of maps $\{f_S:S\in\mathcal{P}_1(\underline{k})\}$ is 
a homotopy equivalence between the punctured cubes.
\end{proof}

Before continuing, notice that the independence complex of a lexicographic product is a polyhedral join, as has been 
pointed out in \citep{okurapoljoin}, in fact 
$\displaystyle\mathcal{F}_0(G\circ H)=\aste{Z}_{\mathcal{F}_0(G)}(\underline{\mathcal{F}_0(H)},\underline{\emptyset})$.
From this we can calculate the homotopy type of $\mathcal{F}_0(G\circ H)$ for some families as applications of the results of section 4, 
we also are able to give some general results by the work done in Section 3.
First we give a formula for the homotopy type of the suspension of $\mathcal{F}_0$ for any lexicographic product in terms of the 
$\mathcal{F}_0$'s of the factors and induced subgraphs of the first factor, this is Corollary \ref{corpolyjoinnullpairs} for the particular 
case of $\mathcal{F}_0(G\circ H)$, notice that 
$\mathrm{lk}_{\mathcal{F}_0(G)}(\sigma)=\mathcal{F}_0\left(G-\bigcup_{v\in\sigma}N[v]\right)$.

\begin{theorem}\label{teosuspf0}
For any graphs $G$ and $H$,
$$\Sigma \mathcal{F}_0(G\circ H)\simeq\bigvee_{\sigma\in \mathcal{F}_0(G)}\sum\left(\mathcal{F}_0\left(G-\bigcup_{v\in\sigma}N[v]\right)*\mathcal{F}_0(H)^{*|\sigma|}\right).$$
\end{theorem}

As immediate consequence we have the following corollary.
\begin{cor}
For any graphs $G$ and $W$,
$$\tilde{H}_q(\mathcal{F}_0(G\circ W))\cong\bigoplus_{\sigma\in \mathcal{F}_0(G)}\tilde{H}_q\left(\mathcal{F}_0\left(G-\bigcup_{v\in\sigma}N[v]\right)*\mathcal{F}_0(W)^{*|\sigma|}\right).$$
\end{cor}

The last theorem gives us an equivalence between the suspensions of two spaces, so it is natural to ask if the formula is true without 
suspending, for some $G$. As we will see, the desuspended formula is true for: cycles, trees, some powers of paths 
and any graph such that its complement has girth at least $4$. So is natural to ask the following question.

\begin{que}
For which graphs $G$, with $\mathcal{F}_0(G)$ connected, is it true that
$$\mathcal{F}_0(G\circ H)\simeq\bigvee_{\sigma\in\mathcal{F}_0(G)}\mathrm{lk}_{\mathcal{F}_0(G)}(\sigma)*\mathcal{F}_0(H)^{*|\sigma|}$$
for all $H$?
\end{que}

Before we proceed with calculations form some particular families, we give a lower bound for the connectivity of the 
independence complex of lexicographic products. The following proposition is Proposition \ref{connpoljoin} for the particular 
case of $\mathcal{F}_0(G\circ H)$.

\begin{prop}
Let $G$ and $H$ be graphs such that $\conn(\mathcal{F}_0(H))\geq k\geq\conn(\mathcal{F}_0(G))\geq0$, then 
$\conn(\mathcal{F}_0(G\circ H))=\conn(\mathcal{F}_0(G))$ and for all $r\leq k+1$
$$\pi_{r}(\mathcal{F}_0(G\circ H))\cong\pi_{r}(\mathcal{F}_0(G))$$
\end{prop}

Notice that the homotopy type of $\mathcal{F}_0(G\circ H)$ does depend on finer details of $G$ than just the homotopy type of 
its independence complex: for example the independence complexes of $P_5$ and $P_6$ have the same 
homotopy type \citep{kozlovdire}  but the ones for the 
corresponding lexicographic products do not have to agree. In \citep{okura} the homotopy type of $\mathcal{F}_0(P_{n}\circ H)$ is computed
when $\mathcal{F}_0(H)$ is homotopy equivalent to a wedge of spheres. In \citep{okurajoinbosq} this was generalized in two different directions: the paper calculates the homotopy type of $\mathcal{F}_0(P_n \circ H)$ for any graph $H$, and the homotopy type of $\mathcal{F}_0(T\circ H)$ when $T$ is a tree and $\mathcal{F}_0(H)$ is 
homotopy equivalent to a wedge of spheres. Here we obtain the homotopy type of $\mathcal{F}_0(T\circ H)$ in general 
as corollary to Theorem \ref{poljointrees}.

\begin{theorem}
Let $T$ be a tree and $H$ any graph. Then 
\begin{itemize}
    \item If $T\cong K_{1,n}$, then $\mathcal{F}_0(T\circ H)\simeq\mathcal{F}_0(H)\sqcup\mathcal{F}_0(H)^{*n}$.
    \item If $T\ncong K_{1,n}$, then
    $$\mathcal{F}_0(T\circ H)\simeq\bigvee_{\sigma\in\mathcal{F}_0(T)}\mathrm{lk}(\sigma)*\mathcal{F}_0(H)^{*|\sigma|} .$$
\end{itemize}
\end{theorem}

For paths, we can give a formula for the homotopy type of $\mathcal{F}_0(P_{n}\circ H)$ using the following polynomials:
$a_0(x,y)=1$, $b_0(x,y)=y$ and $c_0(x,y)=x+2y$
$$a_r(x,y)=(x+y)a_{r-1}(x,y)+yb_{r-1}(x,y)$$
$$b_r(x,y)=(x+y)b_{r-1}(x,y)+yc_{r-1}(x,y)$$
$$c_r(x,y)=(x+y)c_r(x,y)+ya_{r-1}(x,y)$$

\begin{theorem}\label{teopnh}
For any graph $H$
$$\mathcal{F}_0(P_{n}\circ H)\simeq\left\lbrace\begin{array}{cc}
    \mathcal{F}_0(H) &  \mbox{ if }n=1\\
    \mathcal{F}_0(H)\sqcup \mathcal{F}_0(H) & \mbox{ if } n=2\\
    \displaystyle\Sigma\left(\mathcal{F}_0(H)^{\wedge2}\right)\sqcup \mathcal{F}_0(H)& \mbox{ if } n=3\\
    \displaystyle\bigvee_{i,j}\bigvee_{a_{ij}^{(r)}}\left(\Sigma^i\mathcal{F}_0(H)^{*j}\right)& \mbox{ if } n=3r\geq6\\
    \displaystyle\bigvee_{i,j}\bigvee_{b_{ij}^{(r)}}\left(\Sigma^i\mathcal{F}_0(H)^{*j}\right)& \mbox{ if } n=3r+1\geq4\\
    \displaystyle\vee\bigvee_{i,j}\bigvee_{c_{ij}^{(r)}}\left(\Sigma^i\mathcal{F}_0(H)^{*j}\right)& \mbox{ if } n=3r+2\geq5
\end{array}
\right.$$
where $a_{ij}^{(r)},\;b_{ij}^{(r)}\mbox{ and }c_{ij}^{(r)}$ are the coefficients of $x^iy^j$ in $a_r(x,y),\; b_r(x,y),\; c_r(x,y)$ 
respectively. 
\end{theorem}
\begin{proof}
By Theorem \ref{poljointrees} we only need to show that the respecting polynomials count the simplices with non-contractible links. In 
fact we will prove that the coefficient of the polynomials count how many simplicies of size $j$ with link homotopy 
equivalent to $\mathbb{S}^{i-1}$ there are. 

We define $a_0(x,y)=1$, as there is no graph $P_{0}$. For $b_0(x,y)$ we have that $P_{1}$ is just a 
point, so the link of the only vertex is $\emptyset$ and the empty simplex has contractible link, thus $b_0(x,y)=y$. 
For $c_0(x,y)$ we have that $P_{2}$ is 
just an edge and its independence complex is $\mathbb{S}^0$, where the vertices have $\emptyset$ as link and the empty simplex has 
$\mathbb{S}^0$ as link, therefore $c_0(x,y)=x+2y$. For $a_1(x,y)$, the independence complex of $P_3$ is the disjoint 
union of an edge with a vertex, where the links of the edge and of the isolated vertex are $\emptyset$, the links of the vertices in the 
edge are contractible and lastly the link of the empty simplex has $\mathbb{S}^0$ as its homotopy type, thus 
$a_1(x,y)=x+y+y^2=(x+y)a_0(x,y)+yb_0(x,y)$. 

For $\mathcal{F}_0(P_n)$ let $f_{n}(x,y)$ be the polynomial which count how many simplicies of size $j$ such that their links have the homotopy 
type of $\mathbb{S}^{i-1}$ there are in $\mathcal{F}_0(P_n)$. We know that $f_{1}(x,y)=b_0(x,y)$, $f_{2}(x,y)=c_0(x,y)$ and $f_3(x,y)=a_1(x,y)$. 
We take $v_1,v_2,\dots,v_n$ as the vertices of $P_n$. Let $\sigma$ be a simplex of $K=\mathcal{F}_0(P_n)$. There are three possibilities: 
\begin{enumerate}
    \item If $v_2$ is a vertex of $\sigma$, then $\mathrm{lk}_K(\sigma)=\mathrm{lk}_J(\sigma-\{v_2\})$, where 
    $J=\mathcal{F}_0(P_n-N_{P_n}[v_2])\cong\mathcal{F}_0(P_{n-3})$. 
    \item If neither of $v_1,v_2,v_3$ are vertices of $\sigma$, then $v_1,v_2$ are vertices of $\mathrm{lk}_K(\sigma)=\mathcal{F}_0(P_n-N_{P_n}[\sigma])$. 
    If $v_3$ is in $\mathrm{lk}_K(\sigma)$, then 
    $\mathcal{F}_0(P_n-N_{P_n}[\sigma])\simeq\mathcal{F}_0(P_n-v_3-N_{P_n}[\sigma])\cong\Sigma\mathrm{lk}_J(\sigma)$ (see \cite[Lemma 3.2]{engstrom09}).
    If $v_3$ is not in $\mathrm{lk}_K(\sigma)$, then $\mathcal{F}_0(P_n-N_{P_n}[\sigma])\cong\Sigma\mathrm{lk}_J(\sigma)$.
    \item Assume $v_2$ is not in $\sigma$ but at least one of $v_1$ or $v_3$ is in $\sigma$. If $v_1$ is not in $\sigma$, then $v_3$ is in $\sigma$ and 
    $\mathrm{lk}_K(\sigma)$ is contractible because is a cone with $v_1$ as apex vertex. Assume $v_1$ is in $\sigma$, then   
    $\mathrm{lk}_K(\sigma)=\mathrm{lk}_L(\sigma-\{v_1\})$ where $L=\mathcal{F}_0(P_n-N[v_1])\cong\mathcal{F}_0(P_{n-2})$.
\end{enumerate}
From all this we have that $f_{n}(x,y)=(x+y)f_{n-3}(x,y)+yf_{n-2}(x,y)$ and we obtain the result.
\end{proof}

Now we give the generating functions to the polynomials. We define the following series and use the recurrence 
relationship between the polynomials:
$$A(t)=\sum_{r\geq0}a_r(x,y)t^r=1+(x+y)tA(t)+ytB(t)$$
$$B(t)=\sum_{r\geq0}b_r(x,y)t^r=y+(x+y)tB(t)+ytC(t)$$
$$C(t)=\sum_{r\geq0}c_r(x,y)t^r=x+y+(x+y)tC(t)+yA(t)$$
taking $\displaystyle h(t)=\frac{1}{1-(x+y)t}$ we have that
$$A(t)=h(t)\left[1+ytB(t)\right],\;\;B(t)=h(t)\left[y+ytC(t)\right],\;\;C(t)=h(t)\left[x+y+yA(t)\right]$$
Doing the corresponding substitutions and calculations we obtain the following: 
$$A(t)=\frac{h(t)+y^2t(h(t))^2+(xy^2+y^3)t^2(h(t))^3}{1-y^3t^2(h(t))^3}$$
$$B(t)=\frac{yh(t)+(xy+y^2)t(h(t))^2+y^2t(h(t))^3}{1-y^3t^2(h(t))^3}$$
$$C(t)=\frac{(x+y)h(t)+y(h(t))^2+y^3t(h(t))^3}{1-y^3t^2(h(t))^3}$$

\begin{theorem}
$$\mathcal{F}_0(C_n\circ H)\simeq\left\lbrace\begin{array}{cc}
    \displaystyle\bigsqcup_{3}\mathcal{F}_0(H) & \mbox{if } n=3 \\
     \displaystyle\bigsqcup_{3}\mathcal{F}_0(H)^{*2} & \mbox{if } n=4 \\
     \displaystyle\bigvee_{\sigma\in\mathcal{F}_0(C_n)}\mathrm{lk}(\sigma)*\mathcal{F}_0(H)^{*|\sigma|} & \mbox{ for } n\geq5
\end{array}\right.$$
\end{theorem}
\begin{proof}
The cases $n=3,4$ are clear. For $n\geq5$, its a particular case of Theorem \ref{teopoljoincycl}.
\end{proof}

\begin{theorem}\label{theocomplgrith4}
Let $G$ be a graph such that $G^c$ is connected and $g(G^c)\geq4$, then
$$\mathcal{F}_0(G\circ H)\simeq\bigvee_{\sigma\in\mathcal{F}_0(G)}\mathrm{lk}(\sigma)*\mathcal{F}_0(H)^{*|\sigma|}$$
for any graph $H$.
\end{theorem}
\begin{proof}
If $\Delta(G^c)=n-1$, we take $v$ a vertex of maximum degree in $G^c$. Then  $G\cong K_{1,n-1}$ and 
$$\mathcal{F}_0(G\circ H)=\displaystyle\bigcup_{u\in N_G(v)}\mathcal{F}_0(H)^{*2},$$ the homotopy type can be calculate easily with induction and 
Lemma \ref{homocolimpegado}.

We proceed by induction on $m$ the number of edges. For 
$m=0,1,2$ is clear. Assume the theorem is true for any graph $G$ such that $G^c$ is connected, $G^c$ has size less than $m$ and $g(G^c)\geq4$. Let 
$G$ be a graph such that $G^c$ has size $m$  and $g(G^c)\geq4$. We can assume that $\Delta(G^c)<n-1$.
If $\delta(G^c)=1$, take $v$ a degree one vertex and $u$ its only neighbor, then 
$$\mathcal{F}_0(G\circ H)=\mathcal{F}_0(H)^{*2}\cup\aste{Z}_{K}(\underline{\mathcal{F}_0(H)},\underline{\emptyset})$$
where $K$ is the simplicial complex $G^c-v$. If we take
$$J=\mathcal{F}_0(H)^{*2}\cap\aste{Z}_K(\underline{X},\underline{\emptyset})=\aste{D}(u)=\mathcal{F}_0(H),$$
then the inclusions $\mathcal{F}_0(H)\longhookrightarrow\mathcal{F}_0(H)^{*2}$ and 
$\mathcal{F}_0(H)\longhookrightarrow\aste{Z}_{K}(\underline{\mathcal{F}_0(H)},\underline{\emptyset})$
are null-homotopic, by inductive hypothesis and using Lemma \ref{homocolimpegado} we obtain the result. 
Assume that $\delta(G^c)\geq2$, then $G^c$ has at least one cycle and erasing any edge of a cycle does not disconnect $G^c$. Let 
$uv$ an edge which belongs to a vertex-induced cycle $C$, then 
$$\mathcal{F}_0(G\circ H)=\mathcal{F}_0(H)^{*2}\cup\aste{Z}_L(\underline{X},\underline{\emptyset})$$
where $L$ is the simplicial complex $G^c-uv$. As before, we take 
$$J=\aste{D}(uv)\cap\aste{Z}_L(\underline{X},\underline{A})=\mathcal{F}_0(H)\sqcup\mathcal{F}_0(H)$$
Then $J\longhookrightarrow\mathcal{F}_0(H)^{*2}$ and 
$J\longhookrightarrow\aste{Z}_L(\underline{X},\underline{\emptyset})$ are null-homotopic. Therefore 
$$\aste{Z}_{G^c}(\underline{\mathcal{F}_0(H)},\underline{\emptyset})\simeq\aste{D}(uv)\vee\aste{Z}_L(\underline{\mathcal{F}_0(H)},\underline{\emptyset})\vee\Sigma\mathcal{F}_0(H)\vee\Sigma\mathcal{F}_0(H)\vee\mathbb{S}^1.$$
Notice that $G^c\simeq L\vee\mathbb{S}^1$, $\mathrm{lk}_{G^c}(w)\simeq\mathrm{lk}_{L}(w)\vee\mathbb{S}^0$ for $w=u,v$ and 
$\mathrm{lk}_{G^c}(x)\simeq\mathrm{lk}_{L}(x)$ for $x\neq u,v$. The rest follows by inductive hypothesis on $L$.
\end{proof}

The following theorem is just a particular case of Corollary \ref{corsuspcompl}.
\begin{theorem}
Let $H$ and $W=K_l\sqcup G$ graphs with $v_1,\dots,v_l$ the vertices of $K_l$, then
$$\mathcal{F}_0(W\circ H)\simeq\bigvee_{\sigma\in\mathcal{F}_0(G)}\left(\bigvee_{l-1}\Sigma\mathrm{lk}(\sigma)*\aste{D}(\sigma)\vee\bigvee_{i=1}^l\mathrm{lk}(\sigma)*\aste{D}(\sigma\cup\{v_i\})\right).$$
\end{theorem}

\begin{theorem}
For any graph $H$ and $n\geq2r+2$, 
$$\mathcal{F}_0(P_n^r\circ H)\simeq\bigvee_{\sigma\in\mathcal{F}_0(P_n^r)}\mathrm{lk}(\sigma)*\mathcal{F}_0(H)^{*|\sigma|}$$
\end{theorem}
\begin{proof}
The theorem follows form Theorem \ref{poljoinpathpow} and Theorem \ref{theocomplgrith4}.
\end{proof}

We now proceed to study the complexes $\mathcal{F}_d$ for $d\geq1$ for lexicographic products of the form 
$K_{n_1,\dots,n_l}\circ K_r$. First we recall a result for graph joins that will allow us to see that the homotopy type of the second factor 
is not enough to study the homotopy type of these complexes.
Now for any graph $G$, the graph $K_2\circ G$ is also the graph join of two copies of $G$, in \citep{forestfilt} is given a formula 
in general for the join of two graphs for $d\geq1$, here we state this result for the particular case of $K_2\circ G$.

\begin{lem}\citep{forestfilt}\label{lemjoingraf}
Let $G$ be a graph of order $n$. Then:
\begin{enumerate}
    \item $\displaystyle\mathcal{F}_1(K_2\circ G)\simeq\bigvee_2\mathcal{F}_1(G)\vee\bigvee_{n^2-1}\mathbb{S}^1$.
    \item If $\mathcal{F}_0(G)$ is connected, then, for all $d\geq2$,
$$\mathcal{F}_d(K_2\circ G)\simeq\bigvee_{2n-2}\Sigma \mathrm{sk}_{_{d-1}}\mathcal{F}_0(G)\vee\bigvee_{(n-1)^2}\mathbb{S}^2\vee\bigvee_2A$$
where $A=\mathcal{F}_d(G)\cup C(\mathrm{sk}_{_{d-1}} \mathcal{F}_0(G))$.
\end{enumerate}
\end{lem}

Now we can see that in contrast with $\mathcal{F}_0$, the homotopy type of the second factor does not determine the homotopy type of 
$\mathcal{F}_d(G\circ-)$ for $d\geq1$. It is known that $\mathcal{F}_0(P_5)\simeq\mathcal{F}_0(P_6)\simeq\mathbb{S}^1$ 
and $\mathcal{F}_0(P_4)\simeq*$ (see \citep{kozlovdire}), and also that $\mathcal{F}_1(P_5)\simeq\mathcal{F}_1(P_6)\simeq*$ (see \citep{salvetti2018}),
and it is not hard to see that $\mathrm{sk}_1(\mathcal{F}_0(P_5))\simeq\mathbb{S}^1\vee\mathbb{S}^1$ and $\mathrm{sk}_1(\mathcal{F}_0(P_4))\simeq*$. From 
all this and Lemma \ref{lemjoingraf} we have that
$$\mathcal{F}_1(K_2\circ P_5)\simeq\bigvee_{24}\mathbb{S}^1\not\simeq\bigvee_{35}\mathbb{S}^1\simeq\mathcal{F}_1(K_2\circ P_6),$$
$$\mathcal{F}_2(K_2\circ P_4)\simeq\bigvee_{9}\mathbb{S}^2\not\simeq\bigvee_{36}\mathbb{S}^2\simeq\mathcal{F}_2(K_2\circ P_5),$$
and for $d\geq3$,
$$\mathcal{F}_d(K_2\circ P_4)\simeq\bigvee_{9}\mathbb{S}^2\not\simeq\bigvee_{26}\mathbb{S}^2\simeq\mathcal{F}_d(K_2\circ P_5).$$

Until now we only have worked with $\mathcal{F}_0(G\circ H)$, which is a polyhedral join; for $d\geq1$, 
sadly $\mathcal{F}_d(G\circ H)$ is not a polyhedral join but 
$\aste{Z}_{\mathcal{F}_d(G)}\left(\mathrm{sk}_0\Delta^{V(H)},\emptyset\right)$ is a subcomplex. Now, for $H=K_n$ we will make 
calculations for $G$ a complete multipartite graph.

\begin{prop}\label{propk1n}
For any $r$ and $n$,
$$\mathcal{F}_1\left(K_{1,n}\circ K_r\right)\simeq\bigvee_{\binom{r-1}{2}^n}\mathbb{S}^{2n-1}\vee\bigvee_{(nr^2-1)+\binom{r-1}{2}}\mathbb{S}^1;$$
for $2\leq d\leq n-1$,
$$\mathcal{F}_d\left(K_{1,n}\circ K_r\right)\simeq\bigvee_{\binom{r-1}{2}^n}\mathbb{S}^{2n-1}\vee\bigvee_{rf_{_{d-1}}(r,n-1)}\mathbb{S}^d\vee\bigvee_{\binom{r}{2}}\mathbb{S}^1;$$
and for $d=\infty$,
$$\mathcal{F}_\infty\left(K_{1,n}\circ K_r\right)\simeq\bigvee_{\binom{r-1}{2}^n}\mathbb{S}^{2n-1}\vee\bigvee_{r(r-1)^n}\mathbb{S}^n\vee\bigvee_{\binom{r}{2}}\mathbb{S}^1.$$
\end{prop}
\begin{proof}
For $d=1$ the result follows from Lema \ref{lemjoingraf}.

We take $0,1,\dots,n$ as the vertices of $K_{1,n}$ with $0$ the vertex of degree $n$ and $K_r^i$ the copy of $K_r$ corresponding to the vertex $i$.

For $2\leq d\leq n-1$, $\mathcal{F}_d\left(K_{1,n}\circ K_r\right)=X\cup Y\cup Z$ where
$$X=V\left(K_r^0\right)*\aste{Z}_{_{\mathrm{sk}_{d-1}\Delta^{n-1}}}(\mathrm{sk}_0K_r,\emptyset),
\;\;\;Y=\mathcal{F}_d\left(K_r^0\right)\simeq\bigvee_{\binom{r-1}{2}}\mathbb{S}^1,$$
$$Z=\bigast_{i=1}^n\mathcal{F}_d(K_r^i)\simeq\bigvee_{\binom{r-1}{2}^n}\mathbb{S}^{2n-1}.$$
We have that $Y\cap Z=\emptyset$, $X\cap Y=\mathrm{sk}_0Y$ and 
$$X\cap Z=\aste{Z}_{_{\mathrm{sk}_{d-1}\Delta^{n-1}}}(\mathrm{sk}_0K_r,\emptyset)\simeq\bigvee_{f_{_{d-1}}(r,n-1)}\mathbb{S}^{d-1}$$
Once again we compute the homotopy type of union via homotopy pushouts as explained at the end of the preliminaries:
\begin{equation*}
\xymatrix{
\emptyset \ar@{->}[rr]^{\cong} \ar@{->}[dr] \ar@{->}[dd] & & \emptyset \ar@{-}[d] \ar@{->}[rd] & & \\
 & \displaystyle\bigvee_{r-1}\mathbb{S}^0 \ar@{->}[rr]^{\simeq} \ar@{->}[dd] & \ar@{->}[d] & \displaystyle\bigvee_{r-1}\mathbb{S}^0 \ar@{->}[r] \ar@{->}[dd] & \displaystyle\bigvee_{\binom{r-1}{2}}\mathbb{S}^1 \ar@{->}[dd] \\
\displaystyle\bigvee_{f_{_{d-1}}(r,n-1)}\mathbb{S}^{d-1} \ar@{-}[r] \ar@{->}[dr] & \ar@{->}[r]& \displaystyle\bigvee_{\binom{r-1}{2}^n}\mathbb{S}^{2n-1} \ar@{->}[dr] &  & \\
  & \displaystyle\bigvee_{(r-1)f_{_{d-1}}(r,n-1)}\mathbb{S}^{d} \ar@{->}[rr] &  & \hocolim(\mathcal{S}') \ar@{->}[r] & \hocolim(\mathcal{S})
}
\end{equation*}
where $\mathcal{S}'$ is the diagram of the bottom of the cube. Then
$$\hocolim(\mathcal{S}')\simeq\bigvee_{\binom{r-1}{2}^n}\mathbb{S}^{2n-1}\vee\bigvee_{rf_{_{d-1}}(r,n-1)}\mathbb{S}^{d}$$
and the rest follows from this.

Now, for $d=\infty$, $\mathcal{F}_\infty\left(K_{1,n}\circ K_r\right)=X\cup Y\cup Z$ where $Y$ and $Z$ are as before, and
$$X=\aste{Z}_{_{\Delta^n}}(\mathrm{sk}_0K_r,\emptyset)\simeq\bigvee_{(r-1)^{n+1}}\mathbb{S}^n.$$
As before, $Y\cap Z=\emptyset$, $X\cap Y=\mathrm{sk}_0Y$ and 
$$X\cap Z=\aste{Z}_{_{\Delta^{n-1}}}(\mathrm{sk}_0K_r,\emptyset)\simeq\bigvee_{(r-1)^n}\mathbb{S}^{n-1}.$$
Again we use the technique we've been using to compute the homotopy type of the union via homotopy pushouts:
\begin{equation*}
\xymatrix{
\emptyset \ar@{->}[rr]^{\cong} \ar@{->}[dr] \ar@{->}[dd] & & \emptyset \ar@{-}[d] \ar@{->}[rd] & & \\
 & \displaystyle\bigvee_{r-1}\mathbb{S}^0 \ar@{->}[rr] \ar@{->}[dd] & \ar@{->}[d] & \displaystyle\bigvee_{r-1}\mathbb{S}^0 \ar@{->}[r] \ar@{->}[dd] & \displaystyle\bigvee_{\binom{r-1}{2}}\mathbb{S}^1 \ar@{->}[dd] \\
\displaystyle\bigvee_{(r-1)^n}\mathbb{S}^{n-1} \ar@{-}[r] \ar@{->}[dr] & \ar@{->}[r] & \displaystyle\bigvee_{\binom{r-1}{2}^n}\mathbb{S}^{2n-1} \ar@{->}[dr] &  & \\
  & \displaystyle\bigvee_{(r-1)^{n+1}}\mathbb{S}^n \ar@{->}[rr] &  & \hocolim(\mathcal{S}') \ar@{->}[r] & \hocolim(\mathcal{S})
}
\end{equation*}
where $\mathcal{S}'$ again is the diagram of the bottom of the cube. Then 
$$\hocolim(\mathcal{S}')\simeq\bigvee_{\binom{r-1}{2}^n}\mathbb{S}^{2n-1}\vee\bigvee_{r(r-1)^n}\mathbb{S}^n.$$
The result follows from all these.
\end{proof}

\begin{prop}
For any integers $n,m,r\geq2$,
$$\Sigma\mathcal{F}_\infty(K_{n,m}\circ K_r)\simeq\bigvee_{\binom{r-1}{2}^n}\mathbb{S}^{2n}\vee\bigvee_{\binom{r-1}{2}^m}\mathbb{S}^{2m}\vee\bigvee_{a}\mathbb{S}^{n+1}\vee\bigvee_{b}\mathbb{S}^{m+1}\vee\bigvee_{c}\mathbb{S}^3,$$
where $a=m(r-1)^n+m^2(r-1)^m$, $b=n(r-1)^m+n^2(r-1)^n$ and $c=(rn-1)(rm-1)$.
\end{prop}
\begin{proof}
Assume $U=\{u_1,\dots,u_n\}$ and $V=\{v_1,\dots,v_m\}$ are the partition of the vertices of $K_{n,m}$. Taking
$$X=\aste{Z}_{_{\mathcal{F}_\infty(K_{n,m})}}\left(\mathrm{sk}_0K_r,\emptyset\right),\;Y=\aste{Z}_{_{\Delta^U}}\left(K_r,\emptyset\right),\;W=\aste{Z}_{_{\Delta^V}}\left(K_r,\emptyset\right),$$
we have that $\mathcal{F}_\infty(K_{n,m}\circ K_r)=X\cup Y\cup W$. Now, $Y\cap W=X\cap Y\cap W=\emptyset$ and
$$X\cap Y=\aste{Z}_{_{\Delta^U}}\left(\mathrm{sk}_0K_r,\emptyset\right), \;X\cap W=\aste{Z}_{_{\Delta^V}}\left(\mathrm{sk}_0K_r,\emptyset\right).$$
Taking any vertices $u_i\in U$ and $v_j\in V$, we can factor the inclusions to $X$ as
\begin{equation*}
    \xymatrix{
    X\cap Y \ar@{^(->}[r] & \aste{Z}_{_{\Delta^{U\cup\{v_j\}}}}\left(V(K_r),\emptyset\right) \ar@{^(->}[r] & X\\
    X\cap W \ar@{^(->}[r] & \aste{Z}_{_{\Delta^{V\cup\{u_i\}}}}\left(V(K_r),\emptyset\right) \ar@{^(->}[r] & X
    }
\end{equation*}
where the first inclusions are null-homotopic. Therefore 
$$\hocolim\left(X\cup Y \longhookleftarrow X\cap W\longhookrightarrow W\right)\simeq X\cup Y\vee W\vee\Sigma(X\cap W)$$
and 
$$\hocolim\left(X\longhookleftarrow X\cap Y\longhookrightarrow Y\right)\simeq X\vee Y\vee\Sigma(X\cap Y)$$
From where we obtain that 
$$\mathcal{F}_\infty(K_{n,m}\circ K_r)\simeq X\vee Y\vee W\vee\Sigma(X\cap W)\vee\Sigma(X\cap Y)$$
Now, for $\Sigma\mathcal{F}_\infty(K_{n,m}\circ K_r)$ we only need to determine the homotopy type of 
$\Sigma\aste{Z}_{_{\mathcal{F}_\infty(K_{n,m})}}\left(\mathrm{sk}_0K_r,\emptyset\right)$. Now, by Corollary \ref{corpolyjoinnullpairs}
$$\Sigma\aste{Z}_{_{\mathcal{F}_\infty(K_{n,m})}}\left(\mathrm{sk}_0K_r,\emptyset\right)\simeq\Sigma\mathcal{F}_\infty(K_{n,m})\vee\bigvee_{\sigma\in\mathcal{F}_\infty(K_{n,m})-\{\emptyset\}}\left(\bigvee_{(r-1)^{|\sigma|}}\Sigma^{|\sigma|+1}\mathrm{lk}(\sigma)\right).$$
If we take any two vertices from $U$ and any two from $V$, we get a cycle. Therefore $|\sigma\cap U|\leq1$ or $|\sigma\cap V|\leq1$ for 
any simplex $\sigma$. Take $\sigma\in\mathcal{F}_\infty(K_{n,m})$. There are two possibilities:
\begin{itemize}
    \item $\sigma$ is totally contained in $U$ or $V$. Assume $\sigma\subseteq U$. There are two cases:
    \begin{enumerate}
        \item If $|\sigma|=1$, then 
        $$\mathrm{lk}(\sigma)=\left(\mathrm{sk}_0\Delta^V*\Delta^{U-\sigma}\right)\cup\Delta^V$$
        and
        $$\mathrm{lk}(\sigma)\simeq \hocolim\left(*\longleftarrow \mathrm{sk}_0\Delta^V\longrightarrow*\right)\simeq\bigvee_{m-1}\mathbb{S}^1.$$
        \item If $|\sigma|>1$, then
        $$\mathrm{lk}(\sigma)\cong \mathrm{sk}_0\Delta^V*\Delta^{U-\sigma}\simeq\left\lbrace\begin{array}{cc}
            * &  \mbox{ if }|\sigma|<n\\
            \displaystyle\bigvee_{m-1}\mathbb{S}^0 &   \mbox{ if }|\sigma|=n.
        \end{array}\right.$$
    \end{enumerate}
    \item $\sigma\cap U\neq\emptyset\neq\sigma\cap V$. Assume $|\sigma\cap U|=1$. There are three cases:
    \begin{enumerate}
        \item If $2=|\sigma|$, then $\mathrm{lk}(\sigma)=\Delta^{V-\sigma}\sqcup\Delta^{U-\sigma}$ and thus is homotopy equivalent to 
        $\mathbb{S}^0$.
        \item If $2<|\sigma|<m+1$, then $\mathrm{lk}(\sigma)=\Delta^{V-\sigma}$ and therefore is contractible.
        \item If $|\sigma|=m+1$, then $\sigma$ is a maximal simplex and $\mathrm{lk}(\sigma)=\emptyset$.
    \end{enumerate}
\end{itemize}
In \citep{forestfilt} it is show that 
$$\mathcal{F}_\infty(K_{n,m})\simeq\bigvee_{(n-1)(m-1)}\mathbb{S}^2$$
Therefore:
$$\Sigma\aste{Z}_{_{\mathcal{F}_\infty(K_{n,m})}}\left(\mathrm{sk}_0K_r,\emptyset\right)\simeq\bigvee_{c}\mathbb{S}^3\vee\bigvee_{a'}\mathbb{S}^{n+1}\vee\bigvee_{b'}\mathbb{S}^{m+1},$$
where $a'=(m-1)(r-1)^{n}+m(r-1)^{n+1}$, $b'=(n-1)(r-1)^{m}+n(r-1)^{m+1}$ and $c=nm(r-1)^2+n(m-1)(r-1)+m(n-1)(r-1)+(n-1)(m-1)=(rn-1)(rm-1)$.
\end{proof}

\begin{theorem}\label{teomultinf}
For any positive integers $r,n_1,\dots,n_k\geq2$, with $k\geq3$ and $G=K_{n_1,\dots,n_k}\circ K_r$ we have that
$$\Sigma\mathcal{F}_\infty(G)\simeq\bigvee_{i=1}^k\left(\bigvee_{\binom{r-1}{2}^{n_i}}\mathbb{S}^{2n_i}\vee\bigvee_{a_i}\mathbb{S}^{n_i+1}\right)\vee\bigvee_b\mathbb{S}^3\vee\bigvee_{\binom{k-1}{2}}\mathbb{S}^2$$
where 
$$a_i=(r-1)^{n_i}+(t_i+1)(r-1)^{n_i+1}+t_i(r-1)^{n_i},$$
$$b=\sum_{i<j}(n_i-1)(n_j-1)+\sum_{i<j}n_in_j(r-1)^2+\sum_{i=1}^kt_in_i(r-1), \text{ and}$$
$$t_i=\sum_{j\neq i}n_j\;-1.$$
\end{theorem}
\begin{proof}
$$\mathcal{F}_\infty(G)=\aste{Z}_{_{\mathcal{F}_\infty(K_{n_1,\dots,n_k})}}\left(\mathrm{sk}_0K_r,\emptyset\right)\cup\bigsqcup_{i=1}^k\aste{Z}_{_{\Delta^{V_i}}}\left(K_r,\emptyset\right)$$
For all $i$, 
$$\aste{Z}_{_{\mathcal{F}_\infty(K_{n_1,\dots,n_k})}}\left(\mathrm{sk}_0K_r,\emptyset\right)\cap \aste{Z}_{_{\Delta^{V_i}}}\left(K_r,\emptyset\right)=\aste{Z}_{_{\Delta^{V_i}}}\left(\mathrm{sk}_0K_r,\emptyset\right).$$
As in the proposition before, we take $v\in V_j$ with $j\neq i$. Then the inclusion factors as
$$\aste{Z}_{_{\Delta^{V_i}}}\left(\mathrm{sk}_0K_r,\emptyset\right)\longhookrightarrow \aste{Z}_{_{\Delta^{V_i\cup\{v\}}}}\left(\mathrm{sk}_0K_r,\emptyset\right)\longhookrightarrow\aste{Z}_{_{\mathcal{F}_\infty(K_{n_1,\dots,n_k})}}\left(\mathrm{sk}_0K_r,\emptyset\right),$$
and is thus null-homotopic. Therefore,
$$\mathcal{F}_\infty(G)\simeq\aste{Z}_{_{\mathcal{F}_\infty(K_{n_1,\dots,n_k})}}\left(\mathrm{sk}_0K_r,\emptyset\right)\vee\bigvee_{i=1}^k\left(\aste{Z}_{_{\Delta^{V_i}}}\left(K_r,\emptyset\right)\vee\Sigma\aste{Z}_{_{\Delta^{V_i}}}\left(\mathrm{sk}_0K_r,\emptyset\right)\right).$$
For the suspension, as in the last proposition, we have that
$$\Sigma\aste{Z}_{_{\mathcal{F}_\infty(K_{n_1,\dots,n_k})}}\left(\mathrm{sk}_0K_r,\emptyset\right)\simeq\Sigma\mathcal{F}_\infty(K_{n_1,\dots,n_k})\vee\bigvee_{\sigma\in\mathcal{F}_\infty(K_{n_1,\dots,n_k})}\mathrm{lk}(\sigma)*\aste{D}(\sigma)$$
$$\simeq\Sigma\mathcal{F}_\infty(K_{n_1,\dots,n_k})\vee\bigvee_{\sigma\in\mathcal{F}_\infty(K_{n_1,\dots,n_k})}\bigvee_{(r-1)^{|\sigma|}}\Sigma^{|\sigma|+1}\mathrm{lk}(\sigma).$$
Now (see \cite[Corollary 35]{forestfilt}), 
$$\Sigma\mathcal{F}_\infty(K_{n_1,\dots,n_k})\simeq\bigvee_{\binom{k-1}{2}}\mathbb{S}^2\vee\bigvee_{i<j}\Sigma\mathcal{F}_\infty(K_{n_i,n_j}).$$
Because any three vertices $v_i\in V_i$, $v_j\in V_j$ and $v_l\in V_l$, with $i<j<l$, form a cycle, the simplexes with vertices in 
two $V_i,V_j$, with $i\neq j$, have the link as in the corresponding bipartite graph. Therefore, if $\sigma$ is a simplex such that
$\sigma\cap V_i\neq\emptyset\neq\sigma\cap V_j$, for $i\neq j$, and $|\sigma\cap V_i|=1$, there are three cases:
    \begin{enumerate}
        \item If $2=|\sigma|$, then $\mathrm{lk}(\sigma)=\Delta^{V_j-\sigma}\sqcup\Delta^{V_j-\sigma}$ and thus is homotopy equivalent to 
        $\mathbb{S}^0$.
        \item If $2<|\sigma|<n_j+1$, then $\mathrm{lk}(\sigma)=\Delta^{V_j-\sigma}$ and therefore is contractible.
        \item If $|\sigma|=n_j+1$, then $\sigma$ is a maximal simplex and $\mathrm{lk}(\sigma)=\emptyset$.
    \end{enumerate}
Now if $\sigma$ is a simplex such that $\sigma\subseteq V_i$, then 
\begin{enumerate}
        \item If $|\sigma|=1$, then
        $$\mathrm{lk}(\sigma)\cong\left(\left(\bigsqcup_{j\neq i}\mathrm{sk}_0\Delta^{V_j}\right)*\Delta^{V_i-\sigma}\right)\cup\left(\bigsqcup_{j\neq i}\Delta^{V_j}\right)$$
        and
        $$\mathrm{lk}(\sigma)\simeq \hocolim\left(*\longleftarrow \bigsqcup_{j\neq i}\mathrm{sk}_0\Delta^{V_j} \longrightarrow*\right)\simeq\bigvee_{t_i}\mathbb{S}^1.$$
        \item If $|\sigma|>1$, then
        $$\mathrm{lk}(\sigma)\cong\left(\bigsqcup_{j\neq i}\mathrm{sk}_0\Delta^{V_j}\right)*\Delta^{V_i-\sigma}\simeq\left\lbrace\begin{array}{cc}
            * &  \mbox{ if }|\sigma|<n_i\\
            \displaystyle\bigvee_{t_i}\mathbb{S}^0 &   \mbox{ if }|\sigma|=n_i
        \end{array}\right.$$
    \end{enumerate}
The result follows from all these.
\end{proof}

\begin{theorem}
For $1\leq d\leq\min\{n-1,m-1\}$, 
$$\Sigma\mathcal{F}_d(K_{n,m}\circ K_r)\simeq\bigvee_{a_d}\mathbb{S}^2\vee\bigvee_{b_d}\mathbb{S}^3\vee\bigvee_{c_d}\mathbb{S}^{d+1}\vee\bigvee_{\binom{r-1}{2}^n}\mathbb{S}^{2n}\vee\bigvee_{\binom{r-1}{2}^m}\mathbb{S}^{2m},$$
where $a_1=r^2nm-1$, $b_1=c_1=0$ and, for $d\geq2$, $a_d=(n+m)(r-1)$, $b_d=nm(r-1)^2+(m-1)(n-1)$, and
$$c_d=n\binom{m-1}{d}r+m\binom{n-1}{d}r+\left[n\binom{m}{d}+m\binom{n}{d}\right](r-1)^{d+1}+mn\left[\binom{n-2}{d-1}+\binom{m-2}{d-1}\right](r-1)^2$$
$$+\sum_{i=2}^d\left[m\binom{n}{i}\binom{n-i-1}{d-i}+n\binom{m}{i}\binom{m-i-1}{d-i}\right](r-1)^i$$
$$+\sum_{i=3}^d\left[m\binom{n}{i-1}\binom{n-i}{d-i}+n\binom{m}{i-1}\binom{m-i}{d-i}\right](r-1)^i$$
\end{theorem}
\begin{proof}
Assume $U=\{u_1,\dots,u_n\}$ and $V=\{v_1,\dots,v_m\}$ are the partition of the vertices of $K_{n,m}$. Taking
$$X=\aste{Z}_{_{\mathcal{F}_d(K_{n,m})}}\left(\mathrm{sk}_0K_r,\emptyset\right),\;Y=\aste{Z}_{_{\Delta^U}}\left(K_r,\emptyset\right),\;W=\aste{Z}_{_{\Delta^V}}\left(K_r,\emptyset\right),$$
we have that $\mathcal{F}_d(K_{n,m}\circ K_r)=X\cup Y\cup W$. Now, $Y\cap W=X\cap Y\cap W=\emptyset$ and
$$X\cap Y=\aste{Z}_{_{\Delta^U}}\left(\mathrm{sk}_0K_r,\emptyset\right), \;X\cap W=\aste{Z}_{_{\Delta^V}}\left(V(K_r),\emptyset\right).$$
The inclusions $X\cap Y\longhookrightarrow Y$ and $X\cap Z\longhookrightarrow Z$ are null-homotopic, therefore 
$$\mathcal{F}_d(K_{n,m}\circ K_r)\simeq\bigvee_{\binom{r-1}{2}^n}\mathbb{S}^{2n-1}\vee\bigvee_{\binom{r-1}{2}^m}\mathbb{S}^{2m-1}\vee \hocolim(\mathcal{S}),$$
where
\begin{equation*}
    \xymatrix{
    \mathcal{S}: & \ast\sqcup\ast \ar@{<-}[r] & X\cap Y\sqcup X\cap W \ar@{^(->}[r] & X
    }
\end{equation*}
Now, if we define a new complex $K$ from $\mathcal{F}_d(K_{n,m})$ by gluing 
two new simplexes $\Delta^U*\{u_0\}$ and $\Delta^V*\{v_0\}$, with $u_0,v_0$ new vertices, we have that $K\simeq\mathcal{F}_d(K_{n,m})$
and 
$$\hocolim(\mathcal{S})\cong\aste{Z}_{_{K}}\left(\underline{L},\underline{\emptyset}\right),$$
where $L_{u_i}=\mathrm{sk}_0K_r=L_{v_j}$ for $i,j>0$ and $L_{u_0}=pt=L_{v_0}$. Now, 
$$\Sigma\aste{Z}_{_{K}}\left(\underline{L},\underline{\emptyset}\right)\simeq\Sigma\mathcal{F}_d(K_{n,m})\vee\bigvee_{\sigma\in K}\bigvee_{(r-1)^{|\sigma|}}\Sigma^{|\sigma|+1}\mathrm{lk}(\sigma).$$
For any $\sigma$ which contains $u_0$ or $v_0$, $\hat{D}(\sigma)\simeq*$, therefore we only need to know the link for simplexes without
those vertices. Take $U'=U\cup\{u_0\}$ and $V'=V\cup\{v_0\}$
\begin{itemize}
    \item If $\sigma\subseteq U$, there are three possibilities:
    \begin{enumerate}
        \item If $|\sigma|=1$, then $\displaystyle \mathrm{lk}(\sigma)=\Delta^{U'-\{u_i\}}\sqcup \mathrm{sk}_{d-1}\Delta^{V'}\simeq\bigvee_{\binom{m-1}{d}}\mathbb{S}^{d-1}\vee\mathbb{S}^0$.
        \item If $2\leq|\sigma|\leq d$, then
        $$\mathrm{lk}(\sigma)=\left(\bigvee_{m-1}\mathbb{S}^0*\mathrm{sk}_{d-|\sigma|-1}\Delta^{U'-\sigma}\right)\cup\Delta^{U'-\sigma}\simeq\bigvee_{m\binom{n-|\sigma|-1}{d-|\sigma|}}\mathbb{S}^{d-|\sigma|}.$$
        \item If $|\sigma|\geq d+1$, then $\mathrm{lk}(\sigma)=\Delta^{U'-\sigma}\simeq*$.
    \end{enumerate}
    \item Assume $|\sigma\cap U|=1$ and $|\sigma|\geq2$.
    \begin{enumerate}
        \item If $|\sigma|=2$, for $d\geq2$
        $$\mathrm{lk}(\sigma)=\mathrm{sk}_{d-2}\Delta^{U'-\sigma}\sqcup \mathrm{sk}_{d-2}\Delta^{V'-\sigma}\simeq\bigvee_{\binom{n-2}{d-1}}\mathbb{S}^{d-2}\sqcup\bigvee_{\binom{m-2}{d-1}}\mathbb{S}^{d-2}.$$
        \item If $3\leq|\sigma|\leq d$, then 
        $$\mathrm{lk}(\sigma)=\mathrm{sk}_{d-|\sigma|}\Delta^{V'-\sigma}\simeq\bigvee_{\binom{m-|\sigma|}{d-|\sigma|}}\mathbb{S}^{d-|\sigma|}.$$
        \item If $|\sigma|=d+1$, then $\mathrm{lk}(\sigma)=\emptyset$.
    \end{enumerate}
\end{itemize}
Now 
$$\mathcal{F}_d(K_{n,m})\simeq\bigvee_{(n-1)(m-1)}\mathbb{S}^2\vee\bigvee_{n{\binom{m-1}{d}}+m{\binom{n-1}{d}}}\mathbb{S}^d,$$
(see \cite[Theorem 33]{forestfilt}). These the result follows from all these.
\end{proof}

\begin{theorem}\label{teomultford}
For any positive integers $r,n_1,\dots,n_k\geq2$, with $k\geq3$ and $G=K_{n_1,\dots,n_k}\circ K_r$ we have for 
$1\leq d\leq\min\{n_1-1,\dots,n_k-1\}$ that
$$\Sigma\mathcal{F}_d(G)\simeq\bigvee_{a_d}\mathbb{S}^2\vee\bigvee_{b_d}\mathbb{S}^3\vee\bigvee_{c_d}\mathbb{S}^{d+1}\vee\bigvee_{i=1}^k\left(\bigvee_{\binom{r-1}{2}^{n_i}}\mathbb{S}^{2n_i}\right),$$
where $b_1=c_1=0$, 
$$a_1=\sum_{\{i,j\}\in\binom{\underline{k}}{2}}(r^2-2r+2)n_in_j\;+\sum_{i=1}^kn_i(t_i+1)(r-1)\;-k+1,$$
and for $d\geq2$,  $a_d=\frac{(k-1)(k-2)}{2}$,
$$b_d=\sum_{i=1}^kn_i(t_i-k+2)(r-1)+\sum_{\{i,j\}\in\binom{\underline{k}}{2}}(n_in_j(r-1)^2+(n_i-1)(n_j-1))$$
taking 
$$t_i=\sum_{j\neq i}n_j\;-1,\;\;\;\;p_i=\sum_{j\neq i}\binom{n_j-2}{d}$$
$$c_d=\sum_{i=1}^kn_i\left[(t_i+1)\binom{n_i-2}{d-1}+p_i\right](r-1)+\sum_{i=1}^k\sum_{l=2}^d\left[(t_i+1)\binom{n_i}{l}\binom{n_i-l-1}{d-l}(r-1)^l\right]$$
$$+\sum_{\{i,j\}\in\binom{\underline{k}}{2}}\left[n_in_j\left(\binom{n_i-1}{d-2}+\binom{n_j-1}{d-2}\right)(r-1)^2+n_i\binom{n_j-1}{d}+n_j\binom{n_i-1}{d}\right]$$
$$+\sum_{\{i,j\}\in\binom{\underline{k}}{2}}\left[n_i\binom{n_j}{d}+n_j\binom{n_i}{d}\right](r-1)^{d+1}$$
$$+\sum_{\{i,j\}\in\binom{\underline{k}}{2}}\sum_{l=2}^{d-1}\left[n_j\binom{n_i}{l}\binom{n_i-1}{d-l-1}+n_i\binom{n_j}{l}\binom{n_j-1}{d-l-1}\right](r-1)^{l+1}$$
\end{theorem}
\begin{proof}
Assume $V_1,\dots,V_k$ are the partition of the vertices of $K_{n_1,\dots,n_k}$. We have
$$\mathcal{F}_d(G)=\aste{Z}_{_{\mathcal{F}_d(K_{n_1,\dots,n_k})}}\left(\mathrm{sk}_0K_r,\emptyset\right)\cup\bigsqcup_{i=1}^k\left(\bigast_{n_i}K_r\right).$$
For all $1\leq i\leq k$,
$$\aste{Z}_{_{\mathcal{F}_d(K_{n_1,\dots,n_k})}}\left(\mathrm{sk}_0K_r,\emptyset\right)\cap\bigast_{n_i}K_r=\aste{Z}_{_{\Delta^{V_i}}}\left(\mathrm{sk}_0K_r,\emptyset\right),$$
and the inclusion $\displaystyle\aste{Z}_{_{\Delta^{V_i}}}\left(\mathrm{sk}_0K_r,\emptyset\right)\longhookrightarrow\bigast_{n_i}K_r$ is 
null-homotopic. Therefore 
$$\Sigma\mathcal{F}_d(G)=\Sigma\aste{Z}_{_{K}}\left(\underline{L},\underline{\emptyset}\right)\vee\bigvee_{i=1}^k\left(\bigvee_{\binom{r-1}{2}^{n_i}}\mathbb{S}^{2n_i}\right),$$
where:
\begin{itemize}
    \item $K$ is the complex obtain from $\mathcal{F}_d(K_{n_1,\dots,n_k})$ by ading the simplexes $\Delta^{V_i'}$, 
    where $V_i'=V_i\cup\{v_i^0\}$ with $v_i^0$ a new vertice. Clearly $K\simeq\mathcal{F}_d(K_{n_1,\dots,n_k})$.
    \item $L_u=\mathrm{sk}_0K_r$ for any $u\in V(G)$.
    \item $L_{v_i^0}=pt$ for all $i$.
\end{itemize}
As before,
$$\Sigma\aste{Z}_{_{K}}\left(\underline{L},\underline{\emptyset}\right)\simeq\Sigma K\vee\bigvee_{\sigma\in K}\bigvee_{(r-1)^{|\sigma|}}\Sigma^{|\sigma|+1}\mathrm{lk}(\sigma).$$
Since any three vertices form three different sets of the vertex partition give a cycle and for any $\sigma$ which contains a $v_i^0$ 
we have that $\hat{D}(\sigma)\simeq*$, 
the only links we need to determine are those of simplexes contain in one or two sets and that do not contain a vertex $v_i^0$. 

Let $\sigma$ be a simplex such that $\sigma\subseteq V_i$ for some $i$.
\begin{enumerate}
        \item If $|\sigma|=1$, for $d=1$, 
    $\displaystyle \mathrm{lk}(\sigma)=\left(\bigsqcup_{j\neq i}\mathrm{sk}_0\Delta^{V_j}\right)\cup\Delta^{V_i-\sigma}\simeq\bigvee_{t_i+1}\mathbb{S}^0$
        and for $d\geq2$
        $$\mathrm{lk}(\sigma)=\left(\left(\bigsqcup_{j\neq i}\mathrm{sk}_0\Delta^{V_j}\right)*\mathrm{sk}_{d-2}\Delta^{V_i-\sigma}\right)\cup\left(\bigsqcup_{j\neq i}\mathrm{sk}_{d-1}\Delta^{V_j}\right)\cup\Delta^{V_i'-\sigma}.
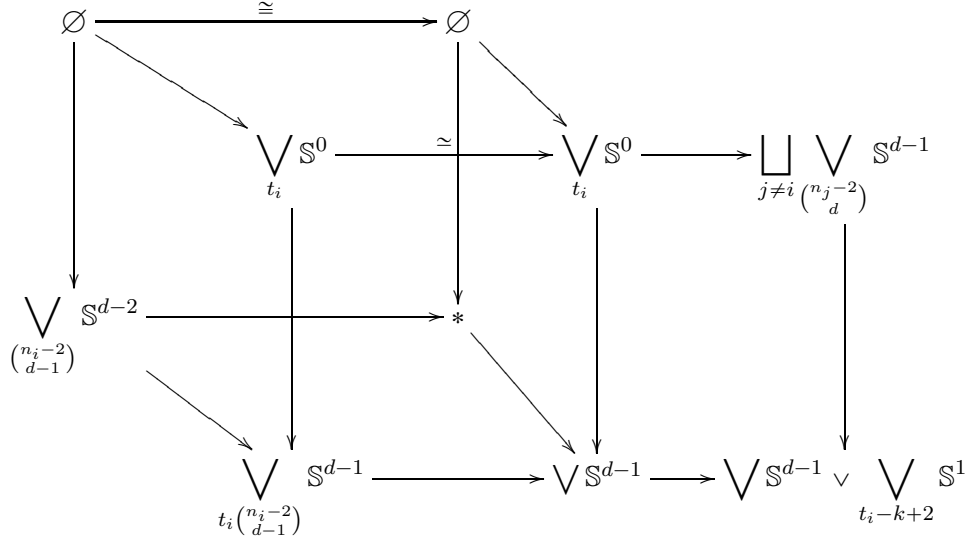
$$
\begin{equation*}
\xymatrix{
\emptyset \ar@{->}[rr]^{\cong} \ar@{->}[dr] \ar@{->}[dd] & & \emptyset \ar@{-}[d] \ar@{->}[rd] & & \\
 & \displaystyle\bigvee_{t_i}\mathbb{S}^0 \ar@{->}[rr]^{\simeq} \ar@{->}[dd] & \ar@{->}[d] & \displaystyle\bigvee_{t_i}\mathbb{S}^0 \ar@{->}[r] \ar@{->}[dd] & \displaystyle\bigsqcup_{j\neq i}\bigvee_{\binom{n_j-2}{d}}\mathbb{S}^{d-1} \ar@{->}[dd] \\
\displaystyle\bigvee_{\binom{n_i-2}{d-1}}\mathbb{S}^{d-2} \ar@{-}[r] \ar@{->}[dr] & \ar@{->}[r] & \ast \ar@{->}[dr] &  & \\
  & \displaystyle\bigvee_{t_i\binom{n_i-2}{d-1}}\mathbb{S}^{d-1} \ar@{->}[rr] &  & \bigvee\mathbb{S}^{d-1} \ar@{->}[r] & \displaystyle\bigvee\mathbb{S}^{d-1}\vee\bigvee_{t_i-k+2}\mathbb{S}^{1}
}
\end{equation*}
        \item If $2\leq|\sigma|\leq d$, then
        $$\mathrm{lk}(\sigma)=\Delta^{V_i'-\sigma}\cup\left(\left(\bigsqcup_{j\neq i}\mathrm{sk}_0\Delta^{V_j}\right)*\mathrm{sk}_{d-|\sigma|-1}\Delta^{V_i-\sigma}\right)\simeq\bigvee_{(t_i+1)\binom{n_i-|\sigma|-1}{d-|\sigma|}}\mathbb{S}^{d-|\sigma|}.$$
        \item If $d+1\leq|\sigma|\leq n_i$, then $\mathrm{lk}(\sigma)=\Delta^{V_i'-\sigma}\simeq*$.
    \end{enumerate}
Let $\sigma$ be a simplex such that $|\sigma\cap V_i|\geq1$ and $|\sigma\cap V_j|=1$ for $i\neq j$.
\begin{enumerate}
    \item If $|\sigma\cap V_i|=1$, for $d=1$ $\mathrm{lk}(\sigma)=\emptyset$ and for $d\geq2$,    $$\mathrm{lk}(\sigma)=\mathrm{sk}_{d-2}\Delta^{V_i}\sqcup\mathrm{sk}_{d-2}\Delta^{V_j}\simeq\bigvee_{\binom{n_i-1}{d-2}}\mathbb{S}^{d-2}\sqcup\bigvee_{\binom{n_j-1}{d-2}}\mathbb{S}^{d-2}.$$
    \item If $|\sigma\cap V_i|=l$ with $2\leq l\leq d-1$, then
    $$\mathrm{lk}(\sigma)=\mathrm{sk}_{d-l-1}\Delta^{V_i}\simeq\bigvee_{\binom{n_i-1}{d-l-1}}\mathbb{S}^{d-l-1},$$
    \item If $|\sigma\cap V_i|=d$, then $\mathrm{lk}(\sigma)=\emptyset$.
\end{enumerate}
Now 
$$\mathcal{F}_d(K_{n_1,\dots,n_k})\simeq\bigvee_{\frac{(k-1)(k-2)}{2}}\mathbb{S}^1\vee\bigvee_{i<j}\mathcal{F}_d(K_{n_i,n_j})$$
(see \citep[Corollary 35]{forestfilt}). The result follows from all this.
\end{proof}

\appendix
\section{Appendix: \texorpdfstring{$n$}{Î£}-Truncations}
Recall that a connected CW-complex $Y$ is $n$-\textit{truncated} if $\pi_q(Y)\cong0$ for all $q>n$ and any base point. 
From the definition we get that if $Y$ is $n$-truncated then it is $(n+k)$-truncated for any $k\geq1$.
The \textit{$n$-truncation or $n$-th Postnikov section} is a functor $\mathrm{P}_n:CW\longrightarrow CW$ such that for any 
connected CW-complex $X$ the space $\mathrm{P}_n X$ is $n$-truncated and there is a $(n+1)$-connected map 
$\Phi_n:X\longrightarrow\mathrm{P}_nX$. As we remarked before, we are not interested in getting a fibration between consecutive sections, however, we do 
want the construction to be functorial and $\Phi_n$ to be natural, to allow us to use it with homotopy colimits. While this approach is not new, we were not able 
to find a reference for Proposition \ref{postnikovhocolim}, so for completeness we prove it and other basic facts about the construction.

For any space $X$, let
$M_n(X)=\mathrm{Map}(\mathbb{S}^n,X)$.
Define for $k\geq1$
$$\mathrm{T}_n^kX=\colim\left(\mathrm{T}_n^{k-1}X\longleftarrow\bigsqcup_{\alpha\in M_{n+k}(\mathrm{T}_n^{k-1}X)}\mathbb{S}^{n+k}\longhookrightarrow\bigsqcup_{\alpha\in M_{n+1}(\mathrm{T}_n^{k-1}X)}\mathbb{D}^{n+k+1}\right)$$
and for $k=0$ we take $\mathrm{T}_n^0=X$. From this we construct the truncation as:
$$\mathrm{P}_nX=\colim\left(X\longhookrightarrow\mathrm{T}_n^1X\longhookrightarrow\mathrm{T}_n^2X\longhookrightarrow\cdots\right)$$
From the construction we get for any $k\geq0$ the pair $(\mathrm{P}_n,\mathrm{T}_n^kX)$ is $(n+k+1)$-connected 
because all the cells of $(\mathrm{P}_nX)-\mathrm{T}_n^kX$ have dimension at least $n+2$ (see \citep[Corollary 4.12]{hatcher}) and
by construction we have that $\pi_{n+k}(\mathrm{T}_n^kX)\cong0$ for any base point and $k\geq1$. We take $\Phi_n$ as the inclusion $\Phi_n:X\longhookrightarrow \mathrm{P}_nX$. If $X$ is connected,
we have that $\mathrm{P}_nX$ is a $n$-truncated space and that $\Phi_n$ is $(n+1)$-connected. Notice that $\mathrm{P}_n\emptyset=\emptyset$ for $X=\emptyset$.

Now, for each map $f:X\longrightarrow Y$ and $\alpha$ in $M_{n+1}(X)$ we have the following diagram:
\begin{equation*}
\xymatrix{
 X \ar@{<-}[r] \ar@{->}[d]^f& \displaystyle\bigsqcup_{\alpha \in M_{n+1}(X)}\mathbb{S}^{n+1} \ar@{^(->}[r] \ar@{->}[d]_{\sqcup f\circ\alpha}& \displaystyle\bigsqcup_{\alpha \in M_{n+1}(X)}\mathbb{D}^{n+2} \ar@{->}[d]\\
Y \ar@{<-}[r] & \displaystyle\bigsqcup_{\beta\in M_{n+1}(Y)}\mathbb{S}^{n+1} \ar@{^(->}[r] & \displaystyle\bigsqcup_{\beta\in M_{n+1}(Y)}\mathbb{D}^{n+2}
}
\end{equation*}
From this diagram we obtain a map $\mathrm{T}_n^1X\longrightarrow \mathrm{T}_n^1Y$. Doing this process for each $k$ we obtain a map 
$\mathrm{P}_n(f):\mathrm{P}_nX\longrightarrow\mathrm{P}_nY$. From the construction is easy to check  
that $\mathrm{P}_n(g\circ f)=\mathrm{P}_n(g)\circ\mathrm{P}_n(f)$. Notice that 
$\mathrm{P}_n(f)\circ\Phi_n$ is the restriction of $\mathrm{P}_n(f)$ to $X$, thus 
$\mathrm{P}_n(f)\circ\Phi_n=\Phi_n\circ f$, \textit{i.e.} $\Phi_n$ is natural. 
From this we get that if $f:X\longrightarrow Y$ is a weak equivalence, then $\mathrm{P}_n f$ also is a weak equivalence 
for all $n$.

\begin{lemapp}\label{postnichar}
The $n$-truncation of a connected CW-complex $X$ is characterised up to homotopy by the following two properties:
\begin{enumerate}
    \item $\mathrm{P}_nX$ is $n$-truncated.
    \item For any $n$-truncated CW-complex $Y$, $-\circ\Phi_n:\mathrm{Map}(\mathrm{P}_nX,Y)\longrightarrow\mathrm{Map}(X,Y)$ 
    is a homotopy equivalence.
\end{enumerate}
\end{lemapp}
\begin{proof}
By construction we know that $\mathrm{P}_nX$ is $n$-truncated. Take $Y$ a $n$-truncated space. The map $\Phi_n$ give us a map 
$$-\circ\Phi_n=\Phi_n^\#:\mathrm{Map}(\mathrm{P}_nX,Y)\longrightarrow\mathrm{Map}(X,Y)$$
Notice that this map is the restriction map. We will show that this map is a weak homotopy equivalence. 
Because $Y$ is $n$-truncated, we have that $\pi_k(Y)\cong0$ for all $k\geq n+1$, thus for any space $B$ and $k\geq n+1$
$$\mathrm{Map}_\ast(\mathbb{S}^k,\mathrm{Map}(B,Y))\cong\mathrm{Map}_\ast(\mathbb{S}^k,\mathrm{Map}_\ast(B_+,Y))\cong\mathrm{Map}_\ast(\mathbb{S}^k\wedge B_+,Y)\simeq\mathrm{Map}_\ast(B_+,\Omega^kY))\simeq\ast.$$
Therefore we only need to check that $\Phi_n^\#$ give us an isomorphism in the homotopy groups for $k\leq n$. Now, for $k\leq n$ and any 
space $B$
$$\Omega^k\mathrm{Map}_\ast(B_+,Y)\simeq\mathrm{Map}_\ast(B_+,\Omega^kY)$$
where $\Omega^kY$ is $(n-k)$-truncated. Thus we only need to check that for any $n$-truncated $Y$, the map $\Phi_n^\#$ gives us a 
bijection for $\pi_0$.

Take a map $f:X\longrightarrow Y$. 
For $k=0$ we have the following commutative diagram
\begin{equation*}
\xymatrix{
X \ar@{<-}[rr] \ar@{->}[rrdd]^f & & \displaystyle\bigsqcup_{\alpha \in M_{n+1}(X)}\mathbb{S}^{n+1} \ar@{^(->}[rr] \ar@{->}[dd]_{\sqcup f\circ\alpha}& & \displaystyle\bigsqcup_{\alpha \in M_{n+1}(X)}\mathbb{D}^{n+2} \ar@{-->}[lldd]\\
 & &  & &\\
 & & Y & &
}
\end{equation*}
where the doted arrow exist because $\pi_{n+1}(Y)\cong0$. From this we get a map $f_1:\mathrm{T}_n^1X\longrightarrow Y$. Now, with the same 
construction we get a map $f_2:\mathrm{T}_n^2X\longrightarrow Y$ and so on. By construction 
$\mathrm{P}_nX=\colim_{k\to\infty}\mathrm{T}_n^k X$, thus we get a map 
$f_\infty:\mathrm{P}_nX\longrightarrow Y$ such that $\Phi_n^\#(f_\infty)=f$.

Now we take maps $f,g:\mathrm{P}_nX\longrightarrow Y$ such that $\Phi_n^\#(f)\simeq\Phi_n^\#(g)$. We will show that for each $k\geq0$
there is a homotopy $H_k:T_n^kX\times I\longrightarrow Y$ between the restrictions of $f$ and $g$ to $T_n^kX$ such that 
${H_{k+1}}_{|_{\mathrm{T}_n^k\times I}}=H_k$. By hypothesis we have $H_0$ and we will show how to get $H_{k+1}$ from $H_k$. We take the 
following commutative diagram
\begin{equation}
\xymatrix{
T_n^kX\times I \ar@{<-}[rr] \ar@{->}[rrdd]^{H_k} & & \displaystyle\bigsqcup_{\alpha \in M_{n+1}(\mathrm{T}_n^kX)}\mathbb{S}^{n+k+1}\times I \ar@{^(->}[rr] \ar@{->}[dd]_{\sqcup H_k\circ(\alpha\times1)}& & \displaystyle\bigsqcup_{\alpha \in M_{n+1}(\mathrm{T}_n^kX)}\mathbb{D}^{n+k+2}\times I \ar@{-->}[lldd]\\
 & &  & &\\
 & & Y & &
}
\end{equation}
where we want the dashed map to exist. To see this, we take the following diagram for each map $\alpha$ in $ M_{n+1}(T_n^kX)$
\begin{equation*}
    \xymatrix{
    \mathbb{S}^{n+k+1}\times\partial I \ar@{^(->}[rr] \ar@{^(->}[dd]& & \mathbb{D}^{n+k+2}\times\partial I \ar@{^(->}[dd] \ar@{->}@/^/[rrddd]^{R}& & \\
     & & & &\\
    \mathbb{S}^{n+k+1}\times I \ar@{^(->}[rr] \ar@{->}@/_/[rrrrd]_{H_k\circ(\alpha\times1)}& & \partial\left(\mathbb{D}^{n+k+2}\times I\right) \ar@{.>}[rrd]& & \\
    & & &  & Y
    }
\end{equation*}
where $R$ in one lid is the restriction of $f$ and in the other lid is the restriction of $g$. Because $Y$ is $n$-truncated, the doted maps 
from last diagram can be extended to a map $\mathbb{D}^{n+k+2}\times I\longrightarrow Y$. These maps give us the desired dashed map in 
diagram $(1)$, from where we get the map 
$$H_{k+1}:\mathrm{T}_n^{k+1}X\times I\longrightarrow Y$$
From all the maps $H_k$ we get the desired homotopy $H:\mathrm{P}_n(X)\times I\longrightarrow Y$.

Now, assume there is an $n$-truncated CW-complex $B$ with a map $f:X\longrightarrow B$ such 
that for any $n$-truncated CW-complex $Y$ the map $f^\#:\mathrm{Map}(B,Y)\longrightarrow\mathrm{Map}(X,Y)$ is a homotopy 
equivalence. If we take $h:\mathrm{P}_nX\longrightarrow B$ the extension of $f$, then we have that for any 
$n$-truncated CW-complex $Y$ we get the following commutative diagram 
\begin{equation*}
    \xymatrix{
    \mathrm{Map}(B,Y) \ar@{->}[rr]^{h^\#}  \ar@{->}@/_{7mm}/[rrrr]_{f^\#} & & \mathrm{Map}(\mathrm{P}_nX,Y)  \ar@{->}[rr]^{\Phi_n^\#} & & \mathrm{Map}( X,Y)
    }
\end{equation*}
where the maps $f^\#$ and $\Phi_n^\#$ are homotopy equivalences. Therefore, for any $n$-truncated CW-complex $Y$ we have that
$h^\#:\mathrm{Map}(B,Y)\longrightarrow\mathrm{Map}(\mathrm{P}_nX,Y)$ is a homotopy equivalence and thus we have a bijection 
$h^\#:\left[B,Y\right]\longrightarrow\left[\mathrm{P}_nX,Y\right]$. Then we can take $g$ in $\mathrm{Map}(B,\mathrm{P}_nX)$ 
such that $[g\circ h]=[1_{\mathrm{P}_nX}]$. Now $h^\#([h\circ g])=[h\circ g\circ h]=[h]=h^\#([1_B])$, therefore 
$h\circ g\simeq1_{\mathrm{P}_nX}$ and $g\circ H\simeq1_B$.
\end{proof}

We want to point out that for $X=\emptyset$ the last lemma is also true.

\begin{corapp}
Let $X$ be a connected CW-complex, then $\mathrm{P}_n\cdots\mathrm{P}_nX\simeq\mathrm{P}_nX$.
\end{corapp}
\begin{proof}
By definition of $\mathrm{P}_n$, we get a map 
$$X\longhookrightarrow\mathrm{P}_nX\longhookrightarrow\cdots\longhookrightarrow\mathrm{P}_n\cdots\mathrm{P}_nX$$
such that 
$$\mathrm{Map}(\mathrm{P}_n\cdots\mathrm{P}_nX,Y)\longrightarrow\cdots\longrightarrow\mathrm{Map}(\mathrm{P}_n X,Y)\longrightarrow\mathrm{Map}(X,Y)$$
is a homotopy equivalence for any $n$-truncated CW-complex $Y$.
\end{proof}

Notice that by construction the homotopy induced by $\Phi_n$ is natural, this together
with Lemma \ref{postnichar} will allows us to prove Proposition \ref{postnikovhocolim}\footnote{The author wishes to thank Omar Antolín for 
telling him about this result and its proof.}. We now restate and prove this proposition.

\begin{propapp}\label{postcoliappen}
Take $\mathcal{I}$ a small category and $\mathcal{X}:\mathcal{I}\longrightarrow Top$ such that 
$\mathcal{X}(i)$ is either empty or a connected CW-complex for any $i$ in $\mathcal{I}$. Then
$$\mathrm{P}_n\left(\hocolim_{i\in\mathcal{I}}\mathcal{X}(i)\right)\simeq\mathrm{P}_n\left(\hocolim_{i\in\mathcal{I}}\mathrm{P}_n(\mathcal{X}(i))\right)$$
\end{propapp}
\begin{proof}
We will show that $\mathrm{P}_n\left(\hocolim_{i\in\mathcal{I}}\mathrm{P}_n(\mathcal{X}(i))\right)$ 
achieve the properties given in Lemma \ref{postnichar}. 
Let $Y$ be a $n$-type, by Lemma \ref{postnichar}
$$\mathrm{Map}\left(\mathrm{P}_n\left(\hocolim_{i\in\mathcal{I}}\mathrm{P}_n(\mathcal{X}(i))\right),Y\right)\simeq\mathrm{Map}\left(\hocolim_{i\in\mathcal{I}}\mathrm{P}_n(\mathcal{X}(i)),Y\right)$$
Because 
$\displaystyle\mathrm{Map}(\hocolim_{i\in\mathcal{I}}\mathcal{Z},-)\cong\holim_{i\in\mathcal{I}^{op}}\mathrm{Map}(\mathcal{Z},-)$ 
for any diagram $\mathcal{Z}:\mathcal{I}\longrightarrow Top$ (see \citep[Proposition 8.5.4]{cubicalhomotopy}), we have that
$$\mathrm{Map}\left(\hocolim_{i\in\mathcal{I}}\mathrm{P}_n(\mathcal{X}(i)),Y\right)\cong\holim_{i\in\mathcal{I}^{op}}\mathrm{Map}\left(\mathrm{P}_n(\mathcal{X}(i)),Y\right)$$
By Lemma \ref{postnichar} and the naturality of $\Phi_n$, we have that
$$\holim_{i\in\mathcal{I}^{op}}\mathrm{Map}\left(\mathrm{P}_n(\mathcal{X}(i)),Y\right)\simeq\holim_{i\in\mathcal{I}^{op}}\mathrm{Map}\left(\mathcal{X}(i),Y\right)$$
Lastly we have
$$\holim_{i\in\mathcal{I}^{op}}\mathrm{Map}\left(\mathcal{X}(i),Y\right)\cong\mathrm{Map}\left(\hocolim_{i\in\mathcal{I}}\mathcal{X}(i),Y\right)$$
\end{proof}

\bibliographystyle{acm}
\bibliography{tipohomotopias}





\end{document}